\documentclass[12pt,a4paper]{amsart}

\usepackage{mathtools}
\usepackage[utf8]{inputenc}
\usepackage[T1]{fontenc}
\usepackage[normalem]{ulem}
\usepackage{cancel}

\usepackage[pdftex]{graphicx}
\usepackage{latexsym}
\usepackage{amsmath,amssymb,amsthm}
\usepackage[shortlabels]{enumitem}   

\usepackage{mathpazo} 
\usepackage{nicefrac} 
\allowdisplaybreaks 

\usepackage{scalerel,stackengine}
\stackMath
\newcommand\reallywidehat[1]{%
\savestack{\tmpbox}{\stretchto{%
  \scaleto{%
    \scalerel*[\widthof{\ensuremath{#1}}]{\kern-.6pt\bigwedge\kern-.6pt}%
    {\rule[-\textheight/2]{1ex}{\textheight}}
  }{\textheight}%
}{0.5ex}}%
\stackon[1pt]{#1}{\tmpbox}
}

\newenvironment{myarray}
  {\arraycolsep=1.4pt%
   \array}
  {\endarray}

\newtheorem{thm}{Theorem}[section]
\newtheorem{lem}[thm]{Lemma}
\newtheorem{definition}[thm]{Definition}

\newtheorem{rem}[thm]{Remark}

\DeclarePairedDelimiter{\abs}{\lvert}{\rvert}
\DeclarePairedDelimiter{\norm}{\lVert}{\rVert}

\newcommand{\C}{\mathbb{C}} 
\newcommand{\R}{\mathbb{R}} 
\newcommand{\Z}{\mathbb{Z}} 
\newcommand{\Zodd}{{\Z_{odd}}}
\newcommand{\N}{\mathbb{N}} 
\newcommand{\Nodd}{{\N_{odd}}}
\newcommand{\T}{{\mathbb T}}

\DeclareMathOperator{\dist}{dist}
\DeclareMathOperator{\trace}{trace}
\DeclareMathOperator{\supp}{supp}

\usepackage[colorlinks=true,urlcolor=blue]{hyperref} 			

\begin{document}

\title[Breather solutions on a periodic metric graph]{Breather solutions for a semilinear Klein-Gordon equation on a periodic metric graph}

\author{Daniela Maier}
\address{D. Maier \hfill\break 
Institute for Analysis, Karlsruhe Institute of Technology (KIT), \hfill\break
D-76128 Karlsruhe, Germany}
\email{daniela.maier11@web.de}

\author{Wolfgang Reichel}
\address{W. Reichel \hfill\break 
Institute for Analysis, Karlsruhe Institute of Technology (KIT), \hfill\break
D-76128 Karlsruhe, Germany}
\email{wolfgang.reichel@kit.edu}

\author{Guido Schneider}
\address{G. Schneider \hfill\break 
Institut f\"ur Analysis, Dynamik und Modellierung, Universit\"at Stuttgart, \hfill\break 
D-70569 Stuttgart, Germany}
\email{guido.schneider@mathematik.uni-stuttgart.de}




\begin{abstract} We consider the nonlinear Klein-Gordon equation
$$
\partial_t^2u(x,t)-\partial_x^2u(x,t)+\alpha u(x,t)=\pm\abs{u(x,t)}^{p-1}u(x,t)
$$
on a periodic metric graph (necklace graph) for $p>1$ with Kirchhoff conditions at the vertices. Under suitable assumptions on the frequency we prove the existence and regularity of infinitely many spatially localized time-periodic solutions (breathers) by variational methods. We compare our results with previous results obtained via spatial dynamics and center manifold techniques. Moreover, we deduce regularity properties of the solutions and show that they are weak solutions of the corresponding initial value problem. Our approach relies on the existence of critical points for indefinite functionals, the concentration compactness principle, and the proper set-up of a functional analytic framework. Compared to earlier work for breathers using variational techniques, a major improvement of embedding properties has been achieved. This allows in particular to avoid all restrictions on the exponent $p>1$ and to achieve higher regularity. 
\end{abstract}

\maketitle


\section{Introduction and results}

We consider the nonlinear Klein-Gordon equation 
\begin{equation}\label{waveeq}
\partial_t^2u(x,t)-\partial_x^2u(x,t)+\alpha u(x,t)=\pm\abs{u(x,t)}^{p-1}u(x,t)
\end{equation}
on the $2\pi$-periodic necklace graph $\Gamma$ with $p\in (1,\infty)$ and $\alpha \geq 0$ together with Kirchhoff conditions at the vertex points. We are interested in breather solutions, i.e., time-periodic, real-valued, and spatially localized solutions. One of the main features of this problem is that the Laplacian $-\frac{d^2}{dx^2}$ with Kirchhoff conditions at the vertices has a spectrum with band-gap structure and in particular spectral gaps occur. 

\medskip

Like in any dynamical system, coherent states (such as e.g. breathers and steady states) are of particular interest. Steady states of \eqref{waveeq} solve an elliptic problem on the necklace graph. Similarly, time-harmonic standing waves of the form $u(x,t)=v(x)e^{i\omega t}$ also solve an elliptic problem of the type $\bigl(-\frac{d^2}{dx^2}+\alpha-\omega^2\bigr)v= \pm |v|^{p-1}v$ on $\Gamma$. Such nonlinear elliptic problems (arising similarly for standing waves of nonlinear Schr\"odinger (NLS) rather than Klein-Gordon equations) have been considered in \cite{pankov_graph,sch_pel}. The outcome is essentially that whenever $\omega^2-\alpha$ is positive and lies in a spectral gap of the Laplacian then solutions $v$ of the nonlinear elliptic profile-equation exist in both the ``$-$'' and the ``$+$'' case. When $\omega^2-\alpha\leq 0$ then the ``$-$'' case has no nontrivial solutions decaying to $0$ at $\pm\infty$ whereas in the ``$+$'' has different types of solutions  of the profile-equation exist (solutions with compact support and solutions with one sign), cf. \cite{sch_pel}. 

\medskip

If instead of time-harmonic standing waves or stationary solutions one looks for real-valued breathers then new difficulties arise. By the requirement of the solutions being real-valued one needs at least two temporal frequencies in the Fourier-decomposition of any solution. But due to the nonlinearity, new frequencies are generated and hence in fact infinitely many frequencies are needed. As a result, the Fourier-modes of standing breathers satisfy an infinitely coupled elliptic system for which the spectral gap condition becomes more delicate than in the monochromatic case. Finding a method to overcome this difficulty is one of the main aspects of this paper.

\medskip

The interest in breather solutions for semilinear wave equations goes back at least as far as the discovery of the explicit breather solution of the sine-Gordon equation $u_{tt}-u_{xx}+\sin u=0$ on $\R\times\R$, cf. \cite{ab_kaup_newell_segur:73}. Similarly, in spatially discrete Fermi-Pasta-Ulam-Tsingou (FPUT) lattices the existence of breathers, cf. \cite{arioli_gazzola:96,james_breathers:09}, is an indication that in certain wave-type equations energy is not always dispersed to infinity. These wave-type equations are, however, quite rare as many nonexistence results show. For example, due to the works \cite{Birnir,Denzler,segur_kruskal}, and more recently \cite{kowalczyk_et_al} it became clear that breathers do not persist in homogeneous nonlinear wave-type equations if the $\sin u$ nonlinearity in the sine-Gordon equation is replaced by a more general nonlinearity $f(u)$ with $f(0)=0, f'(0)>0$. 

\medskip

Therefore, it came as a surprise that heterogeneous wave equations can indeed support breathers. Examples are nonlinear wave equations on discrete lattices, cf. \cite{james_breathers:09,mackay_aubry} or nonlinear wave equations on the real line with $x$-dependent coefficients, cf. \cite{BlaSchneiChiril,hirschreichel}. Whereas the equations considered in \cite{BlaSchneiChiril} and \cite{hirschreichel} are nearly the same, the methods are completely different. The former essentially follows a bifurcation approach using spatial dynamics and center manifold reductions, and produces a family of small breathers bifurcating from $0$ where all of them have the same time-period. The latter approach uses variational methods and finds breathers in a larger parameter regime that do not bifurcate from $0$. In both cases a spectral gap near zero of the wave operator acting on time-periodic functions with a given time-period is vital for the results. Recently, in \cite{plum_reichel} another existence result for vector-valued breathers for a $3+1$-dimensional semilinear curl-curl wave equation with radial symmetry appeared. It is based on ODE-methods and the fact that the breather can be found as a gradient of a spatially radially symmetric function.

\medskip

Yet another way to introduce heterogeneity into a wave equation is to set it up on a quantum graph. Now the heterogeneity stems from the underlying branched structure and not from the equation or the operator. While at first \cite{pankov_graph,sch_pel} standing monochromatic waves were of interest for NLS equations on quantum graphs, more recently Maier \cite{Maier} gave the first existence proof of real-valued breathers for \eqref{waveeq}. Her method first produced spectral gaps in the wave operator by the correct choice of the temporal frequency. Then she used the spatial dynamics point of view and performed a center manifold reduction to show the existence of a homoclinic solution on the center manifold which persisted by going back to the original equation.

\medskip

Our existence result in Theorem~\ref{main} for spatially localized, real-valued time-periodic solutions of \eqref{waveeq} is the main outcome of this paper, and it directly compares to the result from \cite{Maier}. Methodically, we also use tools from the calculus of variations like in \cite{hirschreichel}. However, we have much improved the embedding properties of the underlying Hilbert space $\mathcal{H}$ on which breathers arise as critical points of a suitable functional. In fact, we can cover the whole range of $L^q$-embeddings for $q \in [2,\infty)$, whereas in \cite{hirschreichel} there were still some suboptimal restrictions to $q\in [2,q^\ast)$ with $q^\ast<\infty$.  

\medskip

For the statement of the existence result note that the precise definition of the functions spaces is given in Section~\ref{sec:properties}. In the time direction we denote by $\T$ the $1$-dimensional $4\pi$-periodic torus and by $V$ the set of vertices of $\Gamma$. The space $H^1_c(\Gamma\times\T)$ denotes all $H^1$-functions on $\Gamma\times\T$ such that the $L^2$-traces on $V\times\T$ coincide by approaching from all edges leading to a particular vertex. Since weak derivatives in the spatial direction on a metric graph can only be defined edge-wise and not globally, the Sobolev spaces are also defined edge-wise and thus do not imply continuity of traces at the vertices unless we explicitly require it.

\begin{thm} \label{main} Let $p\in (1,\infty)$ and $\alpha\geq 0$. Then there exist infinitely many distinct weak $2\pi$-time-antiperiodic solutions $u\in H^1_c(\Gamma\times\T)\cap H^2(\Gamma\times\T)$ of \eqref{waveeq} in the sense of Definition~\ref{def_weak_sol}. The solutions are symmetric with respect to the upper/lower-semicircles and satisfy \eqref{waveeq} pointwise a.e. on $\Gamma\times\T$ and the continuity conditions \eqref{kirchhoff_cont} and the Kirchhoff conditions \eqref{kirchhoff_abl} holds at the vertices for almost all $t\in \T$.
\end{thm} 

The results in \cite{Maier} also produce solutions that are symmetric with respect to the upper{/}\-lower-semicircles of the metric graph. It remains an open question whether or not one can find breathers which are not symmetric w.r.t. to the upper/lower-semicircles. The solutions from \cite{Maier} have the precise temporal period $\omega=\frac{k}{2}$, they exist for small $\epsilon>0$ when $\alpha=\frac{k^2}{4}+\epsilon^2$ and bifurcate from $0$ with amplitude $O(\epsilon)$. Multiplicities are not given in this approach. In contrast, we have that infinitely many solutions exist for all values of $\alpha\geq 0$ and in particular do not vanish at the bifurcation points $\alpha=\frac{k^2}{4}$. However, although our solutions are in fact $\frac{2\pi}{\kappa}$ antiperiodic in time where $\kappa\in \Nodd$ is an integer which is chosen sufficiently large in dependance of $A$ for $\alpha \in [0,A]$, we cannot control their minimal period. Finally, while the solutions from \cite{Maier} are exponentially decreasing in space, the spatial localization of our solutions is characterized in weaker form by integrability properties of $u$ up to its second derivatives in space and time.
 
\begin{rem}  
Following \cite{SW} one can generalize the right hand side $f(s)=|s|^{p-1}s$ to more general functions $f=f(x,t,u)$ as follows: assume that
\begin{itemize}
\item[(H1)] $f\colon \Gamma\times\R\times\R\to\R$ is a continuous function, which is $2\pi$-periodic in the first and $2\pi$-antiperiodic as well as symmetric w.r.t. the upper/lower-semicircles in the second variable, with $\vert f(x,t,u)\vert \leq c(1+\vert u\vert^p)$ for some $c>0$ and $p>1$,
\item[(H2)] $f(x,t,u)=o(u)$ as $u\to 0$ uniformly in $x\in\Gamma$, $t\in\T$,
\item[(H3)] $f(x,t,u)$ is odd in $u\in\R$ and $u \mapsto f(x,t,u)/|u|$ is strictly increasing on $(-\infty,0)$ and $(0,\infty)$ for all $x\in\Gamma$ and all $t\in\T$,
\item[(H4)] $\frac{F(x,t,u)}{u^2}\to\infty$ as $u\to\infty$ uniformly in $x\in \Gamma$, $t\in\T$ where $F(x,t,u)\coloneqq \int_0^u f(x,t,v)\,dv$.
\end{itemize}
\end{rem}

\begin{definition} \label{def_weak_sol}
A time-periodic function $u\in H^1_c(\Gamma\times\T)$ is called a weak solution of \eqref{waveeq} if
\begin{equation}
\int_{\Gamma\times\T} \left(\partial_xu\partial_x\phi-\partial_tu\partial_t\phi+\alpha u\phi\mp\abs{u}^{p-1}u\phi\right)\, d(x,t)=0
\end{equation}
holds for every time-periodic $\phi\in H^1_c(\Gamma\times\T)$. 
\end{definition}

Our second theorem describes that the breathers found in Theorem~\ref{main} are also weak solution to the initial value problem with their own induced initial values. 

\begin{thm} \label{prop_ivp} Any weak solution $u$ from Theorem~\ref{main} has the property that $u\in L^2(\T; H^1_c(\Gamma))\cap H^2(\T; L^2(\Gamma))$ and it solves \eqref{waveeq} in the sense that for almost all $t\in (0,\infty)$ 
\begin{equation} \label{ivp}
\int_\Gamma \left(\partial_t^2 u \varphi + \partial_x u\partial_x \varphi +\alpha u\varphi \mp |u|^{p-1}u\varphi\right)\,dx=0, \quad \varphi \in H^1_c(\Gamma) 
\end{equation}
with its own initial values $u(\cdot,0)\in L^2(\Gamma)$, $\partial_t u(\cdot,0)\in L^2(\Gamma)$ in the sense of traces. As a consequence of \eqref{ivp} the continuity conditions \eqref{kirchhoff_cont} and the Kirchhoff conditions \eqref{kirchhoff_abl} hold for almost all $t\in\T$. 
\end{thm}

Finally, we mention that spatial heterogeneity is not a necessary condition for breathers. In fact, it was shown in \cite{mandel_scheider_breather,scheider_breather} that weakly localized breathers exist for an entire class of nonlinearities for constant coefficient nonlinear wave equations in space dimensions higher or equal than $2$.

\medskip

Our main tool for proving existence of breather solutions for \eqref{waveeq} is the use of variational methods for indefinite functionals, cf. \cite{pankov,SW}. Such functionals arise typically for wave equations. On spatially bounded intervals, where the Laplacian has discrete spectrum, breathers for nonlinear wave equations were indeed found variationally, cf. \cite{brezis_coron, brezis_coron_nirenberg, hofer}. These approaches build strongly on the discreteness of the spectrum of the Laplacian, which fails for unbounded spatial domains like the necklace graph. Therefore, the careful building and investigation of the functional framework as described in Section~\ref{sec:functional_framework} is the key to our result.

\medskip

The paper is structured as follows. In Section~\ref{sec:properties} we describe the periodic necklace graph and define on it the Laplacian with Kirchhoff conditions. In Section~\ref{sec:spectrum} we give a characterization of the spectrum of the Kirchhoff Laplacian on the necklace graph and we introduce the Bloch transform and the representation of functions in Bloch variables. Section~\ref{sec:functional_framework} contains the definition of an appropriate Hilbert space $\mathcal{H}$ on which we can define the functional whose critical points will be the breathers. In particular, we show that $\mathcal{H}$ embeds into the whole range of $L^q(\Gamma\times\T_T)$ spaces for $2\leq q<\infty$. This is a major improvement compared to \cite{hirschreichel}. The existence result of Theorem~\ref{main} is given in Section~\ref{sec:existence}, where we rely on variational methods developed in an abstract setting in \cite{SW}. The regularity part of Theorem~\ref{main} and the connection to the initial value problem for \eqref{waveeq} as explained in Theorem~\ref{prop_ivp} is shown in Section~\ref{sec:regularity} by making use of the method of differences, cf. \cite[Chapter 7.11]{gt}. Finally, in the Appendix we show the concentration compactness result of Lemma~\ref{ConccompLemma}. This is nonstandard since our functional framework is not in Sobolev spaces.

\section{Properties of the periodic necklace graph} \label{sec:properties}

The periodic necklace graph is of the form
\begin{equation}
\Gamma=\oplus_{n\in\Z}\Gamma_n,\quad \Gamma_n=\Gamma_{n,0}\oplus\Gamma_{n,+}\oplus\Gamma_{n,-}.
\end{equation}
The edges $\Gamma_{n,0}$ correspond to the horizontal links and the edges $\Gamma_{n,\pm}$ are parallel edges that are visualized by the upper and lower semicircles in Figure \ref{necklace}. 
In order to get a metric graph we identify each edge with a closed interval of length $\pi$. 
In particular, $\Gamma_{n,0}$ is identified isometrically with the interval $I_{n,0}=[2\pi n,2\pi n+\pi]$ and $\Gamma_{n,\pm}$ with $I_{n,\pm}=[2\pi n +\pi,2\pi(n+1)]$. 
Hence, it makes sense to define differential operators on the edges. 
\begin{figure} \label{fig:necklace_graph}
\centering
\includegraphics{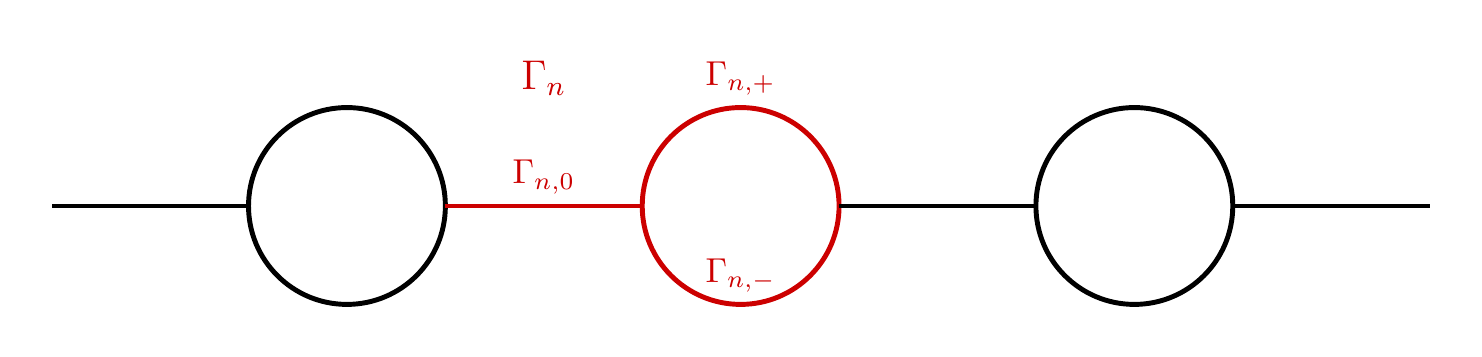}
\caption{Periodic necklace graph}
\label{necklace}
\end{figure}
A function $u:\Gamma\to\C$ can be represented by a triple of functions $u_j:\R\to\C$, $j\in\{0,+,-\}$ with 
\begin{equation}
u|_{\Gamma_{n,j}}=u_j|_{I_{n,j}}\,\mathrm{and}\;\mathrm{supp}(u_j)\subseteq\bigcup_{n\in\Z}I_{n,j}\;\;\mathrm{for}\, j\in\{0,+,-\}.
\end{equation}

\subsection*{Function spaces} A function $u:\Gamma\times\T\to\C$, depending on $x\in\Gamma$ and periodically on $t\in\T$, is continuous and belongs to $C^0(\Gamma\times\T)$, if it is continuous on each edge $\Gamma_{n,j}\times\T$, $j\in\{0,+,-\}$, $n\in\Z$ and the following relations at the vertices are fulfilled for all $n\in \Z$ and all $t\in\T$:
\begin{align}\label{kirchhoff_cont_time}
\begin{myarray}{rlll}
u_0(2\pi n+\pi,t)&=u_+(2\pi n+\pi,t)&=u_-(2\pi n+\pi,t), \\
u_0(2\pi(n+1),t)&=u_+(2\pi(n+1),t)&=u_-(2\pi(n+1),t).
\end{myarray}
\end{align}
Sobolev spaces are defined "edge-wise". If we set $u_{j,n}:=u_{j}|_{I_{j,n}\times\T}$ then we define for $k\in \N_0$ the space
\begin{multline*}
H^k(\Gamma\times\T):=\Bigl\{u:\Gamma\times\T\to\C\;:\;u_{j,n}\in H^k(I_{j,n}\times\T)\,\forall n\in\Z,\,j\in\{0,+,-\}\, \\ \mathrm{with} \sum_{j,n}\norm{u_{j,n}}_{H^k(I_{j,n}\times\T)}^2<\infty\Bigr\}
\end{multline*}
with norm $\|u\|_{H^k}^2:= \sum_{j,n}\norm{u_{j,n}}_{H^k(I_{j,n}\times\T)}^2$. Clearly, $H^0(\Gamma\times\T)=L^2(\Gamma\times\T)$. For $k=1$ we define the coupled Sobolev space $H^1_c(\Gamma\times\T)$ by adding continuity of traces at the vertices, i.e., 
\begin{equation*}
H^1_c(\Gamma\times\T)=\{ u\in H^1(\Gamma\times\T): \eqref{kirchhoff_cont_time} \mbox{ holds in the sense of traces for all } n\in\Z\}.
\end{equation*}
This means that $H^1_c(\Gamma\times\T)$ consists of functions which belong to $H^1$ on every edge$\times\T_T$ and which have additionally traces on $V\times \T$ that are equal independently from which edge one approaches the vertex. We will also need spaces of functions with the additional symmetry between the upper and lower semicircles, i.e., functions $u:\Gamma\times\T\to\C$ with $u_+=u_-$. Spaces of functions with this additional symmetry will be denoted by $H^1_c(\Gamma\times\T)_{\text{symm}}$, $L^2(\Gamma\times\T)_{\text{symm}}$, and $C^0(\Gamma\times\T)_{\text{symm}}$. Whenever we consider spaces of functions, which are independent on time, then we write $H^k(\Gamma)$, $C^0(\Gamma)$, $L^2(\Gamma)$,  $H^1_c(\Gamma)=H^1(\Gamma)\cap C(\Gamma)$, and likewise for their symmetric versions $H^1_c(\Gamma)_{\text{symm}}$, $L^2(\Gamma)_{\text{symm}}$, and $C^0(\Gamma)_{\text{symm}}$. For functions in $H^1_c(\Gamma)$  the conditions \eqref{kirchhoff_cont_time} simplify to true continuity conditions
\begin{align}\label{kirchhoff_cont}
\begin{myarray}{rlll}
u_0(2\pi n+\pi)&=u_+(2\pi n+\pi)&=u_-(2\pi n+\pi), \\
u_0(2\pi(n+1))&=u_+(2\pi(n+1))&=u_-(2\pi(n+1))
\end{myarray}
\end{align}
for all $n\in\Z$.

\subsection*{Laplace operator and Kirchhoff conditions} For functions, which are time-independent, we define the Kirchhoff Laplace operator $\Delta$ as the second derivative operator on the edges with domain $D$ which consists of all functions $u\in H^2(\Gamma)\cap H^1_c(\Gamma)$ together with Kirchhoff conditions at the vertex points (conservation of fluxes at the vertices), i.e., for all $n\in\Z$ 
\begin{align}\label{kirchhoff_abl}
\begin{myarray}{rlll}
\frac{d u_0}{dx}(2\pi n+\pi)&=\frac{du_+}{dx}(2\pi n+\pi)+\frac{du_-}{dx}(2\pi n+\pi), \\
\frac{du_0}{dx}(2\pi(n+1))&=\frac{du_+}{dx}(2\pi(n+1))+\frac{du_-}{dx}(2\pi(n+1)).
\end{myarray}
\end{align}
Then $\Delta: \mathcal{D}(\Delta)\subset L^2(\Gamma)\to L^2(\Gamma)$ is a self-adjoint operator, cf. \cite{gilgschn}, on the domain
$$
\mathcal{D}(\Delta)=\{ u\in H^2(\Gamma)\cap H^1_c(\Gamma): u \mbox{ satisfies } \eqref{kirchhoff_abl}\}.
$$
For time-dependent functions $u\in H^2(\Gamma\times\T)\cap H^1_c(\Gamma\times\T)$, the Kirchoff-conditions are defined as 
\begin{align}\label{kirchhoff_abl_time}
\begin{myarray}{rlll}
\frac{\partial u_0}{\partial x}(2\pi n+\pi,\cdot)&=\frac{\partial u_+}{\partial x}(2\pi n+\pi,\cdot)+\frac{\partial u_-}{\partial x}(2\pi n+\pi,\cdot), \\
\frac{\partial u_0}{\partial x}(2\pi(n+1),\cdot)&=\frac{\partial u_+}{\partial x}(2\pi(n+1),\cdot)+\frac{\partial u_-}{\partial x}(2\pi(n+1),\cdot).
\end{myarray}
\end{align}
in the sense of traces for all $n\in\Z$. 

\section{Spectrum of the Kirchhoff Laplacian} \label{sec:spectrum}

\subsection{Floquet-Bloch spectrum of the Kirchhoff Laplacian}\label{subsec:spectrum}

We recall the spectral properties of the Kirchhoff Laplacian on the necklace graph, cf. \cite{gilgschn}. We begin with the definition of the Bloch waves $g=g(l,x)$ of the Kirchhoff Laplacian with $x\in \Gamma$ and $l\in (-\frac{1}{2},\frac{1}{2}]$ 
as solutions of the quasiperiodic eigenvalue problem
\begin{equation} \label{quasi_ev}
-\frac{d^2}{dx^2} g(l,x)=\lambda(l)g(l,x),\quad x\in\Gamma \mbox{ and } g(l,x+2\pi)=g(l,x)e^{il2\pi}
\end{equation}
with continuity \eqref{kirchhoff_cont} and Kirchhoff \eqref{kirchhoff_abl} conditions at the vertices. 

\medskip

The spectrum of the Kirchhoff Laplacian consists of the discrete and the absolutely continuous part. The discrete part is given by the eigenvalues $m^2$, $m\in\N$, of infinite multiplicity, and the corresponding eigenfunctions $u^{(m)}=(u^{(m)}_0, u^{(m)}_+, u^{(m)}_-)$ are given by periodicity-shifts of the function $u^{(m)}$ defined by $u^{(m)}_0\equiv 0$ and $u^{(m)}_\pm(x)=\pm \sin(mx)$ on $I_{1,\pm}$ and $u^{(m)}_\pm\equiv 0$ on $I_{n,\pm}$ for all $n\in\Z\setminus\{1\}$. The absolutely continuous part is found by using Floquet-Bloch theory and consists of $\cup_{l\in (-\frac{1}{2},\frac{1}{2}]} \cup_{m\in \Z} \lambda_m(l)$, where for each quasiperiodicity parameter $l\in (-\frac{1}{2},\frac{1}{2}]$ the set $\{\lambda_m(l)\}_{m\in \Z}$ is the set of eigenvalues of \eqref{quasi_ev} with continuity \eqref{kirchhoff_cont} and Kirchhoff \eqref{kirchhoff_abl} conditions at the vertices. As an enumeration of $\lambda_m(l)$ we use $m\in \Z$ (the reason for the indexing with $\Z$ is given below). Clearly, the spectrum of the Kirchhoff Laplacian is nonnegative.

\medskip

Instead of computing all eigenvalues $\lambda_m(l)$ one can also obtain the spectrum of the Kirchhoff Laplacian by considering the Hill discriminant, which we will define next. For $\lambda\in \R$ let $\phi_0(\lambda,\cdot)$, $\phi_1(\lambda,\cdot)$ be the fundamental system of solutions of 
\begin{equation} \label{de}
-\frac{d^2}{dx^2} \phi = \lambda \phi \mbox{ satisfying \eqref{kirchhoff_cont}, \eqref{kirchhoff_abl} }
\end{equation}
with 
$$
(\phi_0(\lambda,0), \phi_0'(\lambda,0+))=(1,0), \quad (\phi_1(\lambda,0), \phi_1'(\lambda,0+))=(0,1).
$$
Then the monodromy matrix $M(\lambda)$ is given 
$$
M(\lambda)= 
\begin{pmatrix}
\phi_0(\lambda,2\pi) & \phi_1(\lambda,2\pi) \\
\phi_0'(\lambda,2\pi+) & \phi_1'(\lambda,2\pi+)
\end{pmatrix}
$$
and satisfies for any solution $\phi$ of \eqref{de} that
$$
\begin{pmatrix} \phi \\ \phi' \end{pmatrix}(x+2\pi) = M(\lambda)\begin{pmatrix} \phi \\ \phi' \end{pmatrix}(x), \quad x\in\Gamma.
$$
The Floquet-Bloch theory states that $\lambda\in\sigma(-\frac{d^2}{dx^2})$ if and only if $\mathrm{tr}M(\lambda)\in[-2,2]$ (or equivalently: the $\C$-eigenvalues of $M(\lambda)$ have absolute value $\leq 1$). The function $\lambda \mapsto \mathrm{tr}M(\lambda)$ is called Hill discriminant. In the case of the necklace graph, see \cite{gilgschn}, it takes the particular form
\begin{equation}
\mathrm{tr}M(\lambda)=\frac{1}{4}(9\mathrm{cos}(2\pi\sqrt{\lambda})-1).
\end{equation}
Since a quasiperiodic eigenvalue $\lambda(l)$ can also be characterized by 
$$
\det(M(\lambda)-e^{i l2\pi})=0
$$
we obtain the defining equation for $\lambda(l)$ as 
\begin{equation} \label{quasiperiodic_evalues}
\lambda_m(l)=\left(\frac{1}{2\pi}\arccos\left(\frac{1}{9}(8\mathrm{cos}(2\pi l)+1)\right)+m\right)^2,\quad m\in\Z
\end{equation}
where $m\in \Z$ is therefore the natural index-set for $\lambda_m(l)$. In Figure~\ref{spectrum} we illustrate the band-gap structure of the spectrum by showing the first seven band functions $l\mapsto \lambda_m(l)$ for $m\in \{0,\pm 1, \pm 2,\pm 3\}$. Notice that we have in particular
\begin{equation}
\lambda_m\left(\frac{1}{2}\right)=m^2+\frac{\mathrm{arccos}(-7/9)}{\pi}m+\frac{\mathrm{arccos}(-7/9)^2}{4\pi^2},\quad m\in\Z.
\end{equation}

\begin{figure}
\centering
\includegraphics[scale=0.8]{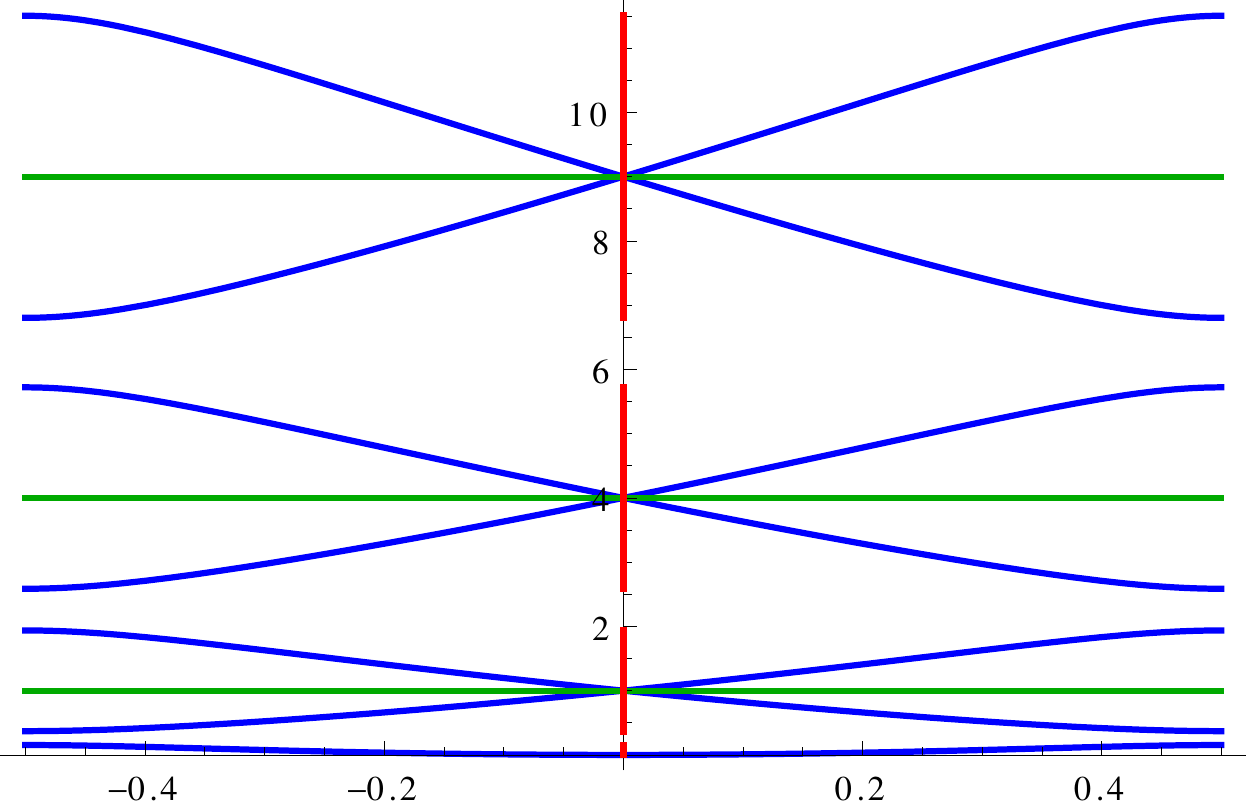}
\caption{The spectral band functions $\lambda_m(l)$ for $m\in \{0,\pm 1, \pm 2,\pm 3\}$ plotted versus the quasiperiodicity-parameter $l\in (-\frac{1}{2},\frac{1}{2}]$. The integer values $m\in \N$ generate the band functions above $m^2$, whereas $m\in -\N$ generate the band functions below $m^2$. The band functions with indices $m, -m$ meet at $m^2$ for $l=0$.}
\label{spectrum}
\end{figure}

\begin{figure}
\centering
\includegraphics{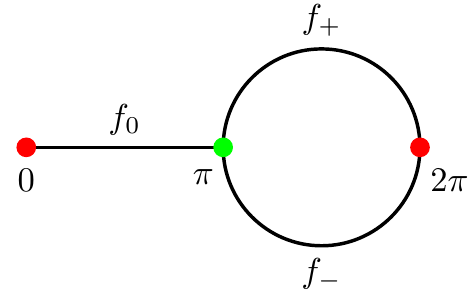}
\caption{Besides the $l$-dependent Kirchhoff vertex conditions we require continuity of the functions $f_{j,n}$ in $x=\pi$ and periodicity between $x=0$ and $x=2\pi$.}
\label{necklace_bloch}
\end{figure}

Alternatively to the quasiperiodic Bloch waves $g=g(l,x)$ one can also consider functions $f=f(l,x)$ by setting 
\begin{equation}
f(l,x) := e^{-ilx}g(l,x),\quad x\in\Gamma, l\in (-\frac{1}{2},\frac{1}{2}]
\end{equation}
Since the quasiperiodicity of $g$ implies periodicity of $f$ we find that the functions $f(l,\cdot)=(f_0,f_+,f_-)(l,\cdot)$ solve the periodic eigenvalue problem
\begin{equation}
-(\frac{d}{dx}+il)^2f(l,x)=\lambda(l)f(l,x),\quad x\in\Gamma \mbox{ and } f(l,x+2\pi)=f(l,x)
\end{equation}
subject to continuity (here we use periodicity)
\begin{equation}\label{kirchhoff_cont_l}
f_0(l,\pi)=f_+(l,\pi)=f_-(l,\pi), \qquad f_0(l,0)=f_+(l,2\pi)=f_-(l,2\pi),
\end{equation}
and $l$-dependent Kirchhoff boundary conditions (here we use periodicity), cf. Figure~\ref{necklace_bloch}
\begin{align}\label{kirchhoff_abl_l}
\begin{myarray}{rlll}
(\frac{d}{dx}+il)f_0(l,\pi)&=(\frac{d}{dx}+il)f_+(l,\pi)+(\frac{d}{dx}+il)f_-(l,\pi), \\
(\frac{d}{dx}+il)f_0(l,0)&=(\frac{d}{dx}+il)f_+(l,2\pi)+(\frac{d}{dx}+il)f_-(l,2\pi).
\end{myarray}
\end{align}
In view of \eqref{kirchhoff_cont_l} we can rephrase \eqref{kirchhoff_abl_l} as 
\begin{equation}\label{kirchhoff_abl_l_old}
\frac{df_0}{dx}(l,\pi)=\frac{df_+}{dx}(l,\pi)+\frac{df_-}{dx}(l,\pi), \qquad
\frac{df_0}{dx}(l,0)=\frac{df_+}{dx}(l,2\pi)+\frac{df_-}{dx}(l,2\pi).
\end{equation}
We shall define $-(\frac{d}{dx}+il)^2$ on a suitable domain, which will be a subspace of 
$$
L^2_{\text{per}}(\Gamma)=\left\{f:\Gamma\to\C\;:\;f_{j,n}\in L^2(I_{j,n}) \mbox{ and } f_{j,n+1}=f_{j,n} \,\forall n\in\Z,\,j\in\{0,\pm\}\right\}
$$
with scalar product  
\begin{equation}
\langle f,h \rangle_{\text{per},2} = \int_{\T_{2\pi}} f \overline{h}\,dx:=\int_0^{\pi} f_0 \overline{h_0}\,dx+\int_{\pi}^{2\pi} f_+ \overline{h_+}\,dx +\int_{\pi}^{2\pi} f_- \overline{h_-}\,dx
\end{equation}
and norm
$$
\|f\|_{\text{per},2}^2 = \langle f,f \rangle_{\text{per},2}
$$
for $f,h \in L^2_{\text{per}}(\Gamma)$. If we set 
$$
\mathcal{D}((\frac{d}{dx}+il)^2):= \left\{f\in L^2_{\text{per}}(\Gamma):\;f_{j,0}\in H^2(I_{j,0}) \forall j\in\{0,\pm\}  \mbox{ and } \eqref{kirchhoff_cont_l}, \eqref{kirchhoff_abl_l_old} \mbox{ hold} \right\}
$$
then by direct verification one sees that $-(\frac{d}{dx}+il)^2: \mathcal{D}((\frac{d}{dx}+il)^2)\subset L^2_{\text{per}}(\Gamma) \to L^2_{\text{per}}(\Gamma)$ is self-adjoint with discrete spectrum consisting of the eigenvalues $\{\lambda_m(l)\}_{m\in \Z}$ given by \eqref{quasiperiodic_evalues}.

\subsection{Bloch transform and Plancherel's identity}

The Fourier transform on the real line diagonalizes the Laplacian with the help of the generalized eigenfunctions $e^{i\xi x}$. A Fourier transform on the necklace graph cannot be properly defined due to periodic branching of the graph. However, the Bloch transform is a suitable substitute that will also diagonalize the Kirchhoff Laplacian on the necklace graph. For a function $u\in L^2(\Gamma)$, its Bloch transform $\tilde u: \R\times\Gamma\to \C$ is defined by 
\begin{equation}
\tilde{u}(l,x):=\sum_{n\in\Z} u(x+2\pi n)e^{-il(x+2\pi n)}
\end{equation}
By construction, $\tilde{u}$ is periodic in $x$ and semi-periodic in $l$, i.e.,
\begin{equation}
\tilde{u}(l,x)=\tilde{u}(l,x+2\pi),\quad \tilde{u}(l,x)=\tilde{u}(l+1,x)e^{ix},
\end{equation} 
and $\mathrm{supp}(\tilde u_j)\subseteq\bigcup_{n\in\Z}I_{n,j}$ for $j\in\{0,+,-\}$.

\begin{lem}\label{planch_lem}
The Bloch transform is an isomorphism between 
\begin{equation*}
L^2(\Gamma)\quad\mathrm{and}\quad L^2((-\tfrac{1}{2},\tfrac{1}{2}], L^2_{\text{per}}(\Gamma))
\end{equation*}
with the inverse given by 
\begin{equation}
u_j(x)=\int_{-\frac{1}{2}}^{\frac{1}{2}}e^{ilx}\tilde{u}_j(l,x) dl,\quad j\in\{0,+,-\}.
\end{equation}
In particular, Plancherel's formula 
\begin{equation}\label{planch}
\langle u,v\rangle_2=\int_{-\frac{1}{2}}^\frac{1}{2} \langle \tilde{u}(l,\cdot), \tilde v(l,\cdot)\rangle_{\rm{per},2} \,dl=:\langle \tilde u,\tilde v\rangle_2
\end{equation}
holds for $u\in L^2(\Gamma)$.
\end{lem}

\begin{rem} \label{planch_ableitungen}
For $u\in H^1_c(\Gamma)$ we have $\norm{\frac{du}{dx}}_2=\norm{\widetilde{\frac{du}{dx}}}_2 = \norm{(\frac{d}{dx}+il)\tilde u}_2$. For $u\in \mathcal D(\Delta)$ we have $\tilde u(\cdot,l) \in \mathcal{D}((\frac{d}{dx}+il)^2)$ for all $l\in (-\tfrac{1}{2},\tfrac{1}{2}]$, and thus, if additionally $v\in L^2(\Gamma)$, then $\langle -\frac{d^2u}{dx^2},v\rangle_2 = \int_{-\frac{1}{2}}^\frac{1}{2}\langle -(\frac{d}{dx}+il)^2 \tilde u(\cdot,l),\tilde v(\cdot,l)\rangle_{\textnormal{per},2}\,dl$. 
\end{rem}

\subsection{Expansion in eigenfunctions of the Laplacian}

Since for every fixed $l\in (-\tfrac{1}{2},\tfrac{1}{2}]$, the operator $-(\frac{d}{dx}+il)^2$ is self-adjoint on $\mathcal{D}(-(\frac{d}{dx}+il)^2)\subset L^2_{\text{per}}(\Gamma)$ we can expand $\tilde{u}(l,\cdot)$ in $L^2$-normalized eigenfunctions $\phi_m(l,\cdot)=\phi_m(l,x)$ of $-(\frac{d}{dx}+il)^2$ as follows
\begin{equation}\label{exp}
\tilde{u}(l,x)=\sum_{m\in\Z}\tilde{u}_m(l)\phi_m(l,x)
\end{equation}
with coefficients $\tilde{u}_m(l)=\int_{\T_{2\pi}} \tilde u(l,x) \overline{\phi_m(l,x)}\,dx\in\C$, $m\in\Z$. 
In particular this expansion diagonalizes the Kirchhoff Laplacian in the following way for $u\in \mathcal{D}(\Delta)$:
\begin{equation}\label{exp'}
-\widetilde{\frac{d^2}{dx^2} u(x)}= -(\frac{d}{dx}+il)^2\tilde{u}(l,x)=\sum_{m\in\Z}\tilde{u}_m(l)\lambda_m(l)\phi_m(l,x).
\end{equation}
Using \eqref{planch} and Remark~\ref{planch_ableitungen} we obtain
\begin{equation}\label{plan_eig}
\norm{u}_2^2 = \int_{-\frac{1}{2}}^{\frac{1}{2}}\sum_{m\in\Z}\abs{\tilde{u}_m(l)}^2\,dl, \quad \norm{\frac{du}{dx}}_2^2 =\int_{-\frac{1}{2}}^{\frac{1}{2}}\sum_{m\in\Z}\abs{\tilde{u}_m(l)}^2\lambda_m(l)dl.
\end{equation}
The eigenfunctions $\phi_m(l,\cdot)$ can be chosen in such a way that they depend in a measurable way on $l$, cf. \cite[XIII.16, Theorem XIII.98]{reed_simon_4}. As a consequence of the min-max principle, the eigenvalues then also depend on $l$ measurably.

The result stated next, which will be used in a crucial way when we study embedding properties cf. Theorem~\ref{hlp}, is about the uniform boundedness of the $2\pi$-periodic eigenfunctions and follows from a result in \cite{komornik}.

\begin{lem}\label{eigen_bound}
The $2\pi$-periodic eigenfunctions $\phi_m(l,\cdot)$ of the operators $-(\frac{d}{dx}+il)^2$, normalized by $\norm{\phi_m(l,\cdot)}_{L^2_\text{per}(\Gamma)}=1$,  are uniformly bounded in $m$,$l$ and $x$, i.e., 
\begin{equation}
\|\phi_m(l,\cdot)\|_{L^\infty(\Gamma)} \leq \frac{12}{\sqrt{\pi}}.
\end{equation}
\end{lem}

\begin{proof} 
On each edge $\Gamma_{n,j}$ with $n\in\Z$ and $j\in \{0,+,-\}$ we see that the function $\psi_m(l,\cdot):= \phi_m(l,\cdot)e^{ilx}|_{\Gamma_{n,j}}$ satisfies $-\psi''=\lambda_m(l)\psi$ on an interval of length $\pi$ with $\lambda_m(l)\geq 0$. Hence we may apply Theorem~2.1. in \cite{komornik} and obtain $\norm{\phi}_{L^\infty(\Gamma_{n,j})}\leq \frac{12}{\sqrt{2\pi}}\norm{\phi}_{L^2(\Gamma_{n,j})}\leq \frac{12}{\sqrt{\pi}}$.
\end{proof}

\section{The functional analytic framework for breathers} \label{sec:functional_framework}

Since we are looking for time-periodic solutions of \eqref{waveeq} it is natural to make an ansatz by Fourier-expanding w.r.t. time 
\begin{equation}
u(x,t)=\sum_{k\in \kappa\Zodd} u_k(x) e^{i\omega kt}
\end{equation}
where $\omega=2\pi/T$ is the temporal frequency, $T$ the temporal period of the ansatz, and $\kappa\in\Zodd$ is an integer to be chosen later. For $u$ to be real we require $\overline{u_k}=u_{-k}$ for all $\kappa\in\Zodd$, and furthermore we require the symmetry $u_{k,+}=u_{k,-}$, i.e., the coefficient functions $u_k$ are symmetric w.r.t. the upper and lower semicircles of the necklace graph.

\medskip

As we shall see, $T$ will be chosen as an integer fraction of $4\pi$ so that all our solutions will be $4\pi$-periodic.  For reasons we will explain later here we have chosen $k\in \kappa\Zodd\subset \Zodd$ instead of $k\in \Z$. This implies that $u$ is $\frac{T}{2\kappa}$-antiperiodic in time, which is compatible with the odd nonlinearity in \eqref{waveeq}. 
The above ansatz decomposes the Klein-Gordon operator $L=\partial_t^2-\partial_x^2+\alpha$ into a family of shifted Kirchhoff Laplacians 
\begin{equation}
L_k :=-\frac{d^2}{dx^2}-\omega^2k^2+\alpha, \quad k\in \kappa\Zodd,
\end{equation}
with domain $\mathcal{D}(L_k)=\mathcal{D}(\Delta)_{\text{symm}}$. Notice that the symmetry between the upper and lower semicircles of the necklace graph is built into the domain of the operator $L_k$. Then, $L_k:\mathcal{D}(L_k)\subset L^2(\Gamma)_{\text{symm}}\to L^2(\Gamma)_{\text{symm}}$ is again a self-adjoint operator. However, now $L_k$ has lost its discrete spectrum, since the discrete eigenvalues (of infinite multiplicity) from Section~\ref{subsec:spectrum} have eigenfunctions which are antisymmetric w.r.t. to the upper/lower-semicircles. Moreover, the quasiperiodic eigenvalues of $L_k$ are given by $\lambda_m(l)-\omega^2k^2+\alpha$ with $\lambda_m(l)$ from \eqref{quasiperiodic_evalues}. For $l\not =0$ the quasiperiodic eigenvalues have multiplicity one and the corresponding eigenfunctions $\phi_m(l,\cdot)$ in the expansion \eqref{exp} are symmetric w.r.t. the upper/lower-semicircles. For $l=0$ the quasiperiodic eigenvalue is in fact a periodic eigenvalue given by $m^2-\omega^2k^2+\alpha$ with $m\in\Z$ and without symmetry restrictions it has multiplicity three. One periodic eigenfunction is  antisymmetric with $\phi_0=0$ on $\Gamma_{n,0}$ and $\phi_\pm(x) = \pm \sin(mx)$ on $\Gamma_{n,\pm}$. Two linearly independent symmetric periodic eigenfunctions are given by $\phi_0(x)=2\sin(mx)$ on $\Gamma_{n,0}$, $\phi_\pm(x) =\sin(mx)$ on $\Gamma_{n,\pm}$ and $\phi_0(x)=\cos(mx)$ on $\Gamma_{n,0}$, $\phi_\pm(x) = \cos(mx)$ on $\Gamma_{n,\pm}$. By applying the symmetry, the operator $L_k$ has only continuous spectrum where the quasiperiodic eigenvalues $\lambda_m(l)-\omega^2k^2+\alpha$ have multiplicity one for $l\not =0$ and multiplicity two for $l=0$. Our choice of $\alpha, \omega$ and $\kappa$ will be such that the following key assumption is satisfied
\begin{equation} \label{spec}
0\not\in \overline{\bigcup_{k\in\kappa\Zodd} \sigma(L_k)} \tag{Spec}
\end{equation}
where $\sigma(L_k)$ is the spectrum of $L_k$. Knowing the structure of the spectrum of $L_k$ in terms of quasiperiodic Bloch eigenvalues the condition \eqref{spec} is equivalent to 
\begin{equation} \label{spec_equiv}
\delta^\ast :=\inf\left(\Bigl|\lambda_m(l)-\omega^2 k^2+\alpha\Bigr|: l\in (-\frac{1}{2},\frac{1}{2}], m\in\Z, k\in \kappa\Zodd\right)>0.
\end{equation}
The associated semi-bounded, closed, Hermitian sesquilinear form of $L_k$ is given by
\begin{equation}
b_{L_k}(v,w)=\int_{\Gamma} v' \overline{w'}+(-\omega^2k^2+\alpha)v\overline{w} \,dx
\end{equation}
with domain $\mathcal{D}(b_{L_k})=H^1_c(\Gamma)_{\text{symm}}$. Functional calculus, cf. \cite[Chapter VIII,6]{reed_simon}, also provides the bilinear form $b_{|L_k|}: H^1_c(\Gamma)_{\text{symm}}\times H^1_c(\Gamma)_{\text{symm}}\to \C$ and \eqref{plan_eig} implies 
\begin{align}
b_{L_k}(v,w) & = \int_{-\frac{1}{2}}^{\frac{1}{2}} \sum_{m\in\Z} \tilde v_m(l)\overline{\tilde w_m(l)} (\lambda_m(l)-\omega^2k^2+\alpha)\,dl,\\
b_{|L_k|}(v,w) &= \int_{-\frac{1}{2}}^{\frac{1}{2}} \sum_{m\in\Z} \tilde v_m(l)\overline{\tilde w_m(l)} |\lambda_m(l)-\omega^2k^2+\alpha|\,dl
\end{align}
for all $v,w \in H^1_c(\Gamma)_{\text{symm}}$.

\begin{definition}
Define the Hilbert space $\mathcal{H}$ by
\begin{equation}
\mathcal{H}=\left\{u=\sum_{k\in \kappa\Zodd}u_ke^{i\omega kt}:\;u_k\in H^1_c(\Gamma)_{\text{symm}},\, \overline{u_k}=u_{-k}\,\forall k\in \kappa\Zodd\,\mathrm{with}\, \|u\|_{\mathcal{H}}<\infty\right\}
\end{equation}
and where the norm is given by 
\begin{equation*}
\|u\|^2_{\mathcal{H}} := \sum_{k\in \kappa\Zodd}b_{\abs{L_k}}(u_k,u_k) = \int_{-\frac{1}{2}}^{\frac{1}{2}} \sum_{m\in\Z}\sum_{k\in\kappa\Zodd} |(\tilde u_k)_m(l)|^2 |\lambda_m(l)-\omega^2k^2+\alpha|\,dl.
\end{equation*}
\end{definition}

\begin{rem} \label{dense} Let us explain why there is a dense subset of $\mathcal{H}$, whose members have compact support. First note that functions $u=\sum_{\substack{k\in \kappa\Zodd\\ |k|\leq K}}u_ke^{i\omega kt}$ with $u_k\in H^1_c(\Gamma)_{\text{symm}}$ and $\overline{u_k}=u_{-k}$ for all $k\in \kappa\Zodd$ with $|k|\leq K$ and $K\in \N$ arbitrary, are dense in $\mathcal{H}$. Since $H^1_c(\Gamma)_{\text{symm}}$-functions can be approximated by functions with compact support, the claim follows.
\end{rem}

On the space $\mathcal{H}$ we define the functional
$$
J[u] = \sum_{k\in\kappa\Zodd} b_{L_k}(u_k,u_k) \mp \tfrac{2}{p+1}\int_0^T\int_\Gamma |u(x,t)|^{p+1} \,dx\,dt
$$
and we shall apply critical point theory from variational calculus to show the existence of $u\in \mathcal{H}$ such that $J'[u]=0$. This function $u$ is then shown to be a weak solution of \eqref{waveeq} in the sense of Definition~\ref{def_weak_sol}. In order to have $J$ as a well-defined Fr\'{e}chet-differentiable functional we need to show the embedding from $\mathcal{H}$ into $L^{p+1}(\Gamma\times\T_T)$, where $\T_T$ is the one dimensional torus of length $T$. This is shown in the subsequent embedding theorem. The abstract condition on the quasiperiodic Bloch eigenvalues is validated afterwards in Lemma~\ref{specific_values} by the particular choice of the admissible values of $\alpha$ and the frequency $\omega$. 

\begin{thm}\label{hlp} Assume that \eqref{spec} holds and suppose $\sigma>1$ and $K\in \kappa\Nodd$ is such that
\begin{equation} \label{convergence}
\int_{-\frac{1}{2}}^{\frac{1}{2}} \sum_{m\in\Z}\sum_{\substack{k \in \kappa\Zodd\\ |k|> K}} |\lambda_m(l)-\omega^2k^2+\alpha|^{-\sigma} \,dl<\infty. 
\end{equation}
Then there is a continuous embedding $\mathcal{H}\to L^q(\Gamma\times \T_T)$ for $q\in [2,\frac{2\sigma}{\sigma-1}]$, and for any finite subgraph $\Gamma_N = \oplus_{|n|\leq N}\Gamma_n$ the embedding $\mathcal{H}\to L^q(\Gamma_N\times \T_T)$ is compact. 
\end{thm}

\begin{proof} First note that \eqref{spec} implies the embedding $\mathcal{H}\to L^2(\Gamma\times \T_T)$. 

\medskip

For $s\in [1,\infty)$ let $L^{s}(\kappa\Zodd\times\Z\times (-\frac{1}{2},\frac{1}{2}])$ be the space of measurable maps $\hat{\tilde u}: \kappa\Zodd\times\Z\times (-\frac{1}{2},\frac{1}{2}]\to\C$ with norm 
$$
\norm{\hat{\tilde u}}_{L^{s}}^{s} := \int_{-\frac{1}{2}}^\frac{1}{2}\sum_{m\in\Z}\sum_{k\in\kappa\Zodd} |(\tilde u_k)_m(l)|^s\,dl.
$$
Then the map 
$$
\mathcal{T}: \hat{\tilde u}\to u \mbox{ with } u(x,t):=\int_{-\frac{1}{2}}^\frac{1}{2} \sum_{m\in\Z}\sum_{k\in\kappa\Zodd} (\tilde u_k)_m(l) \phi_m(l,x) e^{i(lx+\omega kt)}\,dl
$$
is an isometry from $L^2(\kappa\Zodd\times\Z\times(-\frac{1}{2},\frac{1}{2}])$ to $L^2(\Gamma\times \T_T)$ by \eqref{plan_eig}. It extends as a bounded linear operator from $L^1(\kappa\Zodd\times\Z\times(-\frac{1}{2},\frac{1}{2}])$ to $L^\infty(\Gamma\times \T_T)$ if we make use of Lemma~\ref{eigen_bound} and estimate as follows
\begin{align*}
|u(x,t)| & = |\int_{-\frac{1}{2}}^\frac{1}{2} \sum_{m\in\Z}\sum_{k\in\kappa\Zodd} (\tilde u_k)_m(l) \phi_m(l,x) e^{i(lx+\omega kt)}\,dl| \\
& \leq \frac{12}{\sqrt{\pi}} \int_{-\frac{1}{2}}^\frac{1}{2} \sum_{m\in\Z}\sum_{k\in\kappa\Zodd} |(\tilde u_k)_m(l)\,dl|.
\end{align*}
By applying the interpolation theorem of Riesz-Thorin we see that $\mathcal{T}$ extends as a bounded linear operator from $L^{r'}(\kappa\Zodd\times\Z\times (-\frac{1}{2},\frac{1}{2}])$ to $L^r(\Gamma\times\T_T)$ for all $r\in [2,\infty]$. Next we fix $q$ as in the theorem and first split $\hat{\tilde u}=\hat{\tilde u}_1+\hat{\tilde u}_2$ and then $u=u_1+u_2$ as follows
\begin{align*}
u_1(x,t)& =\mathcal{T}(\hat{\tilde u}_1)(x,t) =  \sum_{\substack{k=-K\\ k \in \kappa\Zodd}}^K u_k(x) e^{i\omega kt} \\
u_2(x,t)&=\mathcal{T}(\hat{\tilde u}_2)(x,t) = \sum_{\substack{k \in \kappa\Zodd\\ |k|>K}}^K u_k(x) e^{i\omega kt}=
\int_{-\frac{1}{2}}^\frac{1}{2} \sum_{m\in\Z}\sum_{\substack{k \in \kappa\Zodd\\ |k|> K}} (\tilde u_k)_m(l) \phi_m(l,x) e^{i(lx+\omega kt)}\,dl.
\end{align*}
For the $u_1$-part we use that \eqref{spec} implies that $b_{|L_k|}(u_k,u_k)^\frac{1}{2}$ is equivalent to the $H^1$-norm of $u_k$ so that we have the estimate
\begin{align*}
\|u_1\|_{L^{q^*}(\Gamma\times\T_T)} & \leq T^\frac{1}{q^*}\sum_{\substack{k=-K\\ k \in \kappa\Zodd}}^K \|u_k\|_{L^{q^*}(\Gamma)}\leq C \sum_{\substack{k=-K\\ k \in \kappa\Zodd}}^K \bigl(b_{|L_k|}(u_k,u_k)\bigr)^\frac{1}{2}\\
& \leq \tilde C(K) \left( \sum_{\substack{k=-K\\ k \in \kappa\Zodd}}^K b_{|L_k|}(u_k,u_k)\right)^\frac{1}{2} = \tilde C(K)\|u_1\|_{\mathcal{H}} \leq \tilde C(K)\|u\|_{\mathcal{H}}.
\end{align*}
For the $u_2$-part we use the Riesz-Thorin result for $\mathcal{T}$ with $r=q^*=\frac{2\sigma}{\sigma-1}$ and H\"older's inequality and get 
\begin{align*}
\|u_2&\|_{L^{q^*}(\Gamma\times\T_T)}^{{q^*}'} \\ 
&\leq C\|\hat{\tilde u}_2\|^{{q^*}'}_{L^{{q^*}'}(\kappa\Zodd\times\Z\times (-\frac{1}{2},\frac{1}{2}])}\\
& =C\int_{-\frac{1}{2}}^\frac{1}{2} \sum_{m\in\Z}\sum_{\substack{k \in \kappa\Zodd\\ |k|> K}} |(\tilde u_k)_m(l)|^{{q^*}'} |\lambda_m(l)-k^2\omega^2+\alpha|^\frac{{q^*}'}{2}|\lambda_m(l)-k^2\omega^2+\alpha|^{-\frac{{q^*}'}{2}}\,dl  \\
& \leq C\norm{u}_{\mathcal{H}}^{{q^*}'} \left(\int_{-\frac{1}{2}}^\frac{1}{2} \sum_{m\in\Z}\sum_{\substack{k \in \kappa\Zodd\\ |k|> K}}|\lambda_m(l)-k^2\omega^2+\alpha|^{-\frac{{q^*}'}{2-{q^*}'}}\,dl\right)^\frac{2-{q^*}'}{2},
\end{align*}
where the last integral-sum over $|\lambda_m(l)-k^2\omega^2+\alpha|^{-\sigma}$ is finite due to our assumption \eqref{convergence}. This establishes the continuity of the embedding $\mathcal{T}\circ S: \mathcal{H}\to L^{q^*}(\Gamma\times\T_T)$, where $Su:= \bigl(\tilde u_k)_m(l)\bigr)_{k\in \kappa\Zodd, m\in\Z, l\in(-\frac{1}{2},\frac{1}{2}]}$. The continuity of the embedding extends to all values $q\in [2,q^*]$. Concerning the local compactness of the embedding for $q^*$ we proceed as follows. First, we truncate the embedding $\mathcal{T}$ as follows:
$$
\mathcal{T}_M: \hat{\tilde u}\to u \mbox{ with } u(x,t):=\int_{-\frac{1}{2}}^\frac{1}{2} \sum_{m=-M}^M\sum_{\substack{k=-M \\k \in \kappa\Zodd}}^M (\tilde u_k)_m(l) \phi_m(l,x) e^{i(lx+\omega kt)}\,dl.
$$
One can then verify that in the operator norm $\mathcal{T}_M\to \mathcal{T}$ as $M\to \infty$. Then we inspect the truncated embedding and we show that $\mathcal{T}_M\circ S$ maps $\mathcal{H}$ continuously into $H^1_c(\Gamma\times\T_T)$ which then embeds compactly into $L^q(\Gamma_N\times\T_T)$. Thus the limiting embedding $\mathcal{T}\circ S=\lim_{M\to \infty} \mathcal{T}_M\circ S$ has the same local compactness property. It remains to show that $\mathcal{T}_M\circ S$ maps $\mathcal{H}$ continuously into $H^1_c(\Gamma\times\T_T)$. We use the estimate $|\lambda_m(l)|\leq L(M)$ for all $m\in \Z$ with $|m|\leq M$ and all $l\in (-\frac{1}{2},\frac{1}{2}]$. If $w=\mathcal{T}_M\circ S u$ with $u\in \mathcal{H}$ then 
\begin{align*}
\norm{\partial_x w}^2_{L^2(\Gamma\times \T_T)} &= T\int_{-\frac{1}{2}}^\frac{1}{2} \sum_{m=-M}^M \sum_{\substack{k=-M \\k\in \kappa\Zodd}}^M (\tilde u_k)_m(l)^2 \lambda_m(l)\,dl \\
& \leq L(M)\norm{u}^2_{L^2(\Gamma\times\T_T)}\leq C(M) \norm{u}_{\mathcal H}^2,\\
\norm{\partial_t w}^2_{L^2(\Gamma\times\T_T)} &= T\int_{-\frac{1}{2}}^\frac{1}{2} \sum_{m=-M}^M \sum_{\substack{k=-M \\k \in \kappa\Zodd}}^M (\tilde u_k)_m(l)^2 k^2\omega^2\,dl \\
& \leq M^2\omega^2\norm{u}^2_{L^2(\Gamma\times\T_T)}\leq C(M)\norm{u}_{\mathcal H}^2
\end{align*}
which is what we claimed for $\mathcal{T}_M\circ S$. The continuity of the traces at the vertices is passed on from $\phi_m(l,\cdot)e^{i(l\cdot+\omega kt)}$ to $w$ since on each finite subgraph $w=(\mathcal{T}_M\circ S)(u)$ is a Bochner-integral.
\end{proof}

\begin{lem} \label{sufficient_for_embedding}
A sufficient condition for the convergence of \eqref{convergence} is that for sufficiently large $K\in \kappa\Nodd$  
\begin{equation} \label{sufficient_condition}
    \delta := \inf\left(\Bigl|\sqrt{\lambda_m(l)}-\sqrt{\omega^2 k^2-\alpha}\Bigr|: l\in (-\frac{1}{2},\frac{1}{2}], m\in\Z, k\in \kappa\Zodd \mbox{ with } |k|>K\right)>0.
\end{equation}
\end{lem}

\begin{proof}
We assume $K$ so large that  $\omega^2 |k|^2\geq \alpha+\alpha_0$ for some $\alpha_0>0$ and all $k\in \kappa\Zodd$ with $|k|>K$.

\medskip

Next we make the following observation: if $\delta:=\dist(A,\Z)>0$ and $\sigma>1$ then there exists $B=B(\sigma)>0$ such that $\sum_{m\in \Z} |A-m|^{-\sigma} \leq B+\delta^{-\sigma}$. In particular $B$ does not depend on $A$ or $\delta$. This can be seen as follows: let $m_0\in \Z$ be such that $\delta=|A-m_0|$. From the inequality $\delta\leq \frac{1}{2}<1\leq |m-m_0|$ for $k\in \Z\setminus\{m_0\}$ we find that $|A-m|\geq |m-m_0|- |A-m_0|=|m-m_0|-\delta\geq \frac{1}{2}|m-m_0|$ for $m\in\Z\setminus\{m_0\}$. This implies 
$$
\sum_{m\in \Z} |A-m|^{-\sigma} \leq \delta^{-\sigma} +\underbrace{\sum_{m\in \Z\setminus\{m_0\}} 2^\sigma|m-m_0|^{-\sigma}}_{=:B}
$$
as claimed. 

\medskip

From \eqref{quasiperiodic_evalues} with 
\begin{equation} \label{al} 
a(l)= \frac{1}{2\pi}\arccos\bigl(\frac{1}{9}(8\mathrm{cos}(2\pi l)+1)\bigr)\in [0,\frac{1}{2})
\end{equation} 
we use the representation $\sqrt{\lambda_m(l)}=m+a(l)$ if $m\geq 0$, $\sqrt{\lambda_m(l)}=-m-a(l), m<0$ and in particular $\sqrt{\lambda_m(l)}\geq \frac{1}{2}$ for $m\not =0$. Next we estimate the sum in \eqref{convergence} with the help of the above observation as follows:
\begin{align*}
\int_{-\frac{1}{2}}^{\frac{1}{2}} \sum_{m\in\Z} & \sum_{\substack{k\in\kappa\Zodd \\ |k|>K}} |\lambda_m(l)-\omega^2k^2+\alpha|^{-\sigma} \,dl\\
\leq & \int_{-\frac{1}{2}}^{\frac{1}{2}} \sum_{m\in\N_0}\sum_{\substack{k\in\Zodd\\ |k|>K}} \Bigl|m+a(l)-\sqrt{\omega^2 k^2-\alpha}\Bigr|^{-\sigma} \Bigl|\underbrace{m+a(l)}_{\geq 0}+\sqrt{\omega^2 k^2-\alpha}\Bigr|^{-\sigma}  \,dl \\
& + \int_{-\frac{1}{2}}^{\frac{1}{2}} \sum_{m\in\N}\sum_{\substack{k\in\Zodd\\ |k|>K}} \Bigl|m-a(l)-\sqrt{\omega^2 k^2-\alpha}\Bigr|^{-\sigma} \Bigl|\underbrace{m-a(l)}_{\geq 1/2\geq 0}+\sqrt{\omega^2 k^2-\alpha}\Bigr|^{-\sigma}  \,dl \\
\leq & 2 (\delta^{-\sigma}+B) \sum_{\substack{k\in \Zodd\\ |k|>K}} (\underbrace{\omega^2 k^2-\alpha}_{\geq \alpha_0})^{-\sigma/2}<\infty
\end{align*}
which proves the claim.
\end{proof}

By the previous lemma we can finally give specific values of $\alpha$ and $\omega$ for which both \eqref{spec} holds and \eqref{convergence} converges.

\begin{lem} \label{specific_values}
Let $\omega=\frac{k_0}{2}$, $\alpha\in [0,A]$ for some $A>0$ and $k_0\in \Nodd$. Then we can choose $\kappa\in \Nodd$ sufficiently large (depending on $A$) such that both \eqref{spec} is true and the embedding $\mathcal H\to L^q(\Gamma\times [0,4\pi])$ holds for every $q\in [2,\infty)$. 
\end{lem}

\begin{proof} 
We use the representation $\lambda_m(l)= (m+a(l))^2$ for $m\in\Z$ and $a(l)$ as in \eqref{al}. In particular $0<\delta_0 := 1-2a(\frac{1}{2})\leq 1-2a(l)$ for all $l\in (-\frac{1}{2},\frac{1}{2}]$. Next let us consider the expression $\sqrt{\omega^2 k^2-\alpha}$ for $k\in \kappa\Nodd$ (w.l.o.g. we can restrict to positive integers $k$). By the mean value theorem there is $\xi\in (1-\frac{4\alpha}{k^2k_0^2},1)\subset(1-\frac{4}{\kappa^2},1)$ such that 
$$
\sqrt{\omega^2 k^2-\alpha} = \frac{kk_0}{2}\sqrt{1-\frac{4\alpha}{k^2k_0^2}} = \frac{kk_0}{2} - \frac{\alpha}{k k_0\sqrt{\xi}} \geq \frac{kk_0}{2} - \frac{\alpha}{\kappa\sqrt{\xi}} \geq \frac{kk_0}{2}-\frac{\delta_0}{2},
$$
provided $\kappa\in \Nodd$ is chosen sufficiently large. In order to show \eqref{spec_equiv}, which is equivalent to \eqref{spec}, and \eqref{sufficient_condition} from Lemma~\ref{sufficient_for_embedding} we distinguish two cases for $m\in\Z$.

\medskip

\noindent
\emph{Case 1:} $|m|\geq \frac{kk_0}{2}$. Since $m$ is an integer and $kk_0$ is an odd number,  we see that $|m|\geq \frac{kk_0+1}{2}$ and consequently $|m+a(l)|\geq |m|-a(l)= |m|-\frac{1}{2}+\frac{1}{2}-a(l)\geq \frac{kk_0}{2}+\frac{\delta_0}{2}$ so that 
$$
|\lambda_m(l)-\omega^2k^2+\alpha| = |m+a(l)|^2-\omega^2k^2+\alpha \geq \frac{k^2k_0^2}{4}+kk_0\frac{\delta_0}{2}+\frac{\delta_0^2}{4}-\frac{k^2k_0^2}{4}+\alpha \geq \kappa \frac{\delta_0}{2}.
$$
Likewise 
\begin{align*}
2|\sqrt{\lambda_m(l)}-\sqrt{\omega^2 k^2-\alpha}| & \geq 2|m+a(l)|-kk_0 +\frac{2\alpha}{kk_0\sqrt{\xi}} \\
& \geq 2|m|-2a(l)-kk_0 \geq 1-2a(l)\geq \delta_0. 
\end{align*}

\noindent
\emph{Case 2:} $|m|<\frac{kk_0}{2}$. Then $|m|\leq \frac{kk_0-1}{2}$ and 
\begin{align*}
|\lambda_m(l)-\omega^2k^2+\alpha| & = \frac{k^2k_0^2}{4}-\alpha - (m+a(l))^2 \geq \frac{k^2k_0^2}{4}-\alpha - (|m|+a(l))^2 \\
& = \frac{k^2k_0^2}{4}-\alpha- m^2-2|m|a(l)-a(l)^2  \\
& \geq kk_0-\frac{1}{4}-(kk_0-1)a(l)-\alpha-a(l)^2  \\
& = kk_0(1-a(l))-\frac{1}{4}+\underbrace{a(l)}_{\geq 0}-\underbrace{a(l)^2}_{\leq 1/4}-\alpha \\
& \geq kk_0(1-a(l)) - \frac{1}{2}-\alpha \geq \kappa\delta_0- \frac{1}{2}-A\geq \kappa \frac{\delta_0}{2}
\end{align*}
provided $\kappa\in\Nodd$ is chosen large enough. Similarly
\begin{align*}
2|\sqrt{\lambda_m(l)}-\sqrt{\omega^2 k^2-\alpha}| & \geq kk_0 -\frac{\alpha}{kk_0\sqrt{\xi}} -2|m|-2a(l) \\
& \geq 1-2a(l) -\frac{\alpha}{\kappa\sqrt{\xi}} \geq \frac{\delta_0}{2} 
\end{align*}
provided $\kappa\in \Nodd$ is sufficiently large.
\end{proof}

Due to the assumption \eqref{spec} there are two spectral projections $P_k^\pm: H^1_c(\Gamma)_{\text{symm}}\to H^1_c(\Gamma)_{\text{symm}}^\pm$ and two further projections $P^\pm: \mathcal{H}\to \mathcal{H}^\pm$. By Lemma~\ref{specific_values} we know that 
$$
\lambda_m(l)-\omega^2k^2+\alpha  >  0 \Leftrightarrow |m| \geq \frac{|k|k_0+1}{2} 
$$
and 
$$
\lambda_m(l)-\omega^2k^2+\alpha  <  0 \Leftrightarrow |m| \leq \frac{|k|k_0-1}{2}.
$$
Therefore for $v\in H^1_c(\Gamma)_{\text{symm}}$ with 
$$
v(x)=\int_{-\frac{1}{2}}^\frac{1}{2} \sum_{m\in \Z} \tilde v_m(l)\phi_m(l,x)e^{ilx}\,dl
$$
we have 
\begin{align*}
v^+ := P_k^+ v = & \int_{-\frac{1}{2}}^\frac{1}{2}\sum_{\substack{m\in\Z\\ |m|\geq \frac{|k|k_0+1}{2} }} \tilde v_m(l) \phi_m(l,x) e^{ilx}\,dl, \\
v^- := P_k^- v = & \int_{-\frac{1}{2}}^\frac{1}{2}\sum_{\substack{m\in\Z\\ |m|\leq \frac{|k|k_0-1}{2} }} \tilde v_m(l) \phi_m(l,x) e^{ilx}\,dl
\end{align*}
and thus $H^1_c(\Gamma)_{\text{symm}}=H^1_c(\Gamma)_{\text{symm}}^+\oplus H^1_c(\Gamma)_{\text{symm}}^-$. Likewise, for $u\in\mathcal{H}$ with 
$$
u(x,t)=\sum_{k\in\kappa\Zodd} u_k(x) e^{ik\omega }
$$
we can now explicitly write down the spectral projections by
$$
u^\pm := P^\pm u = \sum_{k\in\kappa\Zodd} P_k^\pm u_k(x) e^{ik\omega t} 
$$
and hence $\mathcal{H}^\pm = P^\pm \mathcal{H}$ with $P^\pm$ having the above explicit representation of the projection operators.

\section{Existence part of Theorem~\ref{main}} \label{sec:existence}

In this part we closely follow the existence proof given in \cite{hirschreichel}. However, compared to \cite{hirschreichel} Theorem~\ref{hlp} and Lemma~\ref{specific_values} provide improved embeddings which allow us to simplify the arguments. Let $J\colon \mathcal{H} \to \R$ be given by
\begin{align*}
J[u] \coloneqq J_0[u]\mp J_1[u] 
\end{align*}
with
\begin{align*}
J_0[u]\coloneqq B(u,u):=\sum_{k\in\kappa\Zodd} b_{L_k}(u_k,u_k), \quad J_1(\tilde{u})\coloneqq  \frac{2}{p+1}\int_0^T\int_\Gamma |u(x,t)|^{p+1}\, dx\,dt.
\end{align*}
Here, $\mp$ in the definition of $J$ corresponds to the sign $\pm$ in \eqref{waveeq}. By Lemma~\ref{specific_values} and consequently Theorem~\ref{hlp} the functional $J$ is well-defined on $\mathcal{H}$. We will find a time-periodic solution of \eqref{waveeq} as a minimizer of the functional $J$ on the generalized Nehari manifold defined as 
\begin{align*}
\mathcal{M}\coloneqq \{u\in \mathcal{H}\setminus \mathcal{H}^-: J'[u](u)=0 \text{ and } J'[u](v)=0 \text{ for all } v\in \mathcal{H}^-\}.
\end{align*}
The idea to use $\mathcal{M}$ as a constraint for minimization goes back to \cite{nehari1,nehari2} and in its generalized form to \cite{pankov}. It was then systematically explored in an abstract form in \cite{SW}. One finds that $\mathcal{M}$ is a natural constraint in the sense that it does not generate Lagrange multipliers. We make use of an abstract result from \cite{SW} that guarantees the existence of minimizing sequence for an indefinite functional $J$ on $\mathcal{M}$. For this purpose we check next the assumptions of Theorem~35, Chapter~4 from \cite{SW}.

\medskip

We first treat the ''$+$''-case in \eqref{waveeq}. At the end of this section we explain how the ''$-$''-case can be treated. Moreover, for $u\in\mathcal{H}$ we set
\begin{align*}
\mathcal{H}(u)\coloneqq \R^+ u \oplus \mathcal{H}^- = \R^+ u^+ \oplus \mathcal{H}^-,
\end{align*}
where $\R^+= [0,\infty)$. Finally, let $S$ denote the unit ball in $\mathcal{H}$ and define $S^+\coloneqq S\cap \mathcal{H}^+$. By standard calculations (compare Proposition~1.12 in \cite{Willem}) we deduce $J\in C^1(\mathcal{H})$ and 
\begin{align*}
J'[u](v) = J_0'[u](v)-J_1'[u](v)= B(u,v)-\int_D |u|^{p-1} u \overline{v} \,dx\,dt.
\end{align*}

We start verifying the assumption $(B_1)$, (i) and (ii) of Theorem~35 in \cite{SW}.

\begin{lem} \label{FormelNichtlin1}
The following statements hold true:
\begin{itemize}
\item[(a)] $J_1$ is weakly lower semicontinuous, 
\begin{align} \label{BeinsersteVor}
J_1[0]=0 \quad \text{ and } \quad \frac{1}{2}J_1'[u](u)> J_1[u]>0 \text{ for } u\neq 0.
\end{align}
\item[(b)] $\lim_{u\to 0} \frac{J_1'[u]}{\|u\|_{\mathcal{H}}}=0$ and $\lim_{u\to 0} \frac{J_1[u]}{\|u\|^2_{\mathcal{H}}}=0$.
\item[(c)] For a weakly compact set $U\subset \mathcal{H}\setminus\{0\}$ we have $\lim_{s\to\infty} \frac{J_1[s u]}{s^2} = \infty$ uniformly w.r.t. $u\in U$.
\end{itemize}
\end{lem}

\begin{proof}
(a) The inequality in \eqref{BeinsersteVor} follows from $p>1$. The weak lower-semicontinuity of $J_1$ follows form convexity and the continuity of the embedding $\mathcal{H}\to L^{p+1}(\Gamma\times\T_T)$. 

\smallskip

(b) Both claims are immediate by the embedding provided by Theorem~\ref{hlp}.

\smallskip

(c) Let $U\subset \mathcal{H}\setminus\{0\}$ be weakly compact. To prove the claim it is sufficient to show that for every sequence $(u_n)_{n\in \N}$ in $U$ and every sequence $s_n\to \infty$ we have $\liminf_{n\in \N} \frac{J_1[s_n u_n]}{s_n^2} = \infty$. Up to a subsequence we have that $u_n\to u$ a.e. in $\Gamma\times\T_T$ as $n\to \infty$ and $u\not=0$ on a set $A\subset \Gamma\times\T_T$ of positive measure. Clearly  
$$
\lim_{n\to \infty} \frac{|s_n u_n(x,t)|^{p+1}}{s_n^2 u_n(x,t)^2} = \infty  \quad \mbox{ a.e. on } A 
$$
so that by Fatou's Lemma 
$$
\liminf_{n \in N} \frac{J_1[s_n u_n]}{s_n^2} \geq  \liminf_{n\in N} \int_A \frac{|s_nu_n(x,t)|^{p+1}}{s_n^2 u_n(x,t)^2} u_n(x,t)^2\,d(x,t) = \infty.
$$
\end{proof}

Assumption $(B_2)$ of Theorem~35 in \cite{SW} is guaranteed by the next result.

\begin{lem} \label{FormelNichtlin2}
The following statements hold true:
\begin{itemize}
\item[(a)] For each $w\in\mathcal{H}\setminus \mathcal{H}^-$ there exists a unique nontrivial critical point $m_1(w)$ of $J|_{\mathcal{H}(w)}$. Moreover, $m_1(w)\in \mathcal{M}$ is the unique global maximizer of $J|_{\mathcal{H}(w)}$ as well as $J[m_1(w)]>0$.
\item[(b)] There exists $\delta>0$ such that $\| m_1(w)^+\|_{\mathcal{H}} \geq \delta$ for all $w\in \mathcal{H}\setminus \mathcal{H}^-$.
\end{itemize}
\end{lem}

\begin{proof}
(a) We can directly follow the lines of proof of Proposition~39 in \cite{SW}.

\smallskip

(b) First, consider $v\in\mathcal{H}^+$. Then we have $\lim_{v\to 0} \frac{J[v]}{\|v\|^2_{\mathcal{H}}} =1$
due to Lemma~\ref{FormelNichtlin1}~(b). Thus there is $\rho_0>0$ s.t. $J[v]\geq \frac{1}{2}\|v\|^2_{\mathcal{H}}$ for all $v\in\mathcal{H}^+$ with $\|v\|_{\mathcal{H}} \leq \rho_0$. Hence for $\rho \in (0,\rho_0)$ we find $\eta=\frac{\rho^2}{2}$ with $J[v]\geq \eta$ for all $v\in\mathcal{H}^+$ with $\|v\|_{\mathcal{H}}=\rho$. Now, let $w\in \mathcal{H}\setminus\mathcal{H}^-$. Due to the structure of $J$ we infer that 
\begin{align} \label{tildewteileins}
\|m_1(w)^+\|^2_{\mathcal{H}} \geq J[m_1(w)].
\end{align} 
Since $m_1(w)$ is the maximizer of $J|_{\mathcal{H}(w)}$ we conclude
\begin{align} \label{tildewteilzwei}
J[m_1(w)]\geq J\left[\rho \frac{w^+}{\|w^+\|_{\mathcal{H}}}\right] \geq \eta
\end{align}
and the combination of \eqref{tildewteileins} and \eqref{tildewteilzwei} finishes the proof of part (b).
\end{proof}

\begin{lem} \label{NichtLin5}
Any Palais-Smale sequence $(u_n)_{n\in\N}$ of $J|_{\mathcal{M}}$ is bounded.
\end{lem}

\begin{proof} The following proof is similar to the proof of Theorem~40 in \cite{SW}. 

\smallskip

\noindent
{\em Step 1:} Suppose for contradiction that $(u_n)_{n\in \N}$ is an unbounded Palais-Smale sequence for $J$. By selecting a subsequence we may assume that $\|u_n\|_{\mathcal{H}}\to \infty$ and that $v_n := u_n/\|u_n\|_{\mathcal{H}}$ has the property that $v_n\rightharpoonup v$ as $n\to \infty$. Note that 
\begin{equation} \label{rescale}
0 \leq \frac{J[u_n]}{\|u_n\|_{\mathcal{H}}^2} = \|v_n^+\|_{\mathcal{H}}^2 - \|v_n^-\|_{\mathcal{H}}^2 - \frac{J_1[\|u_n\|_{\mathcal{H}} v_n]}{\|u_n\|_{\mathcal{H}}^2}.
\end{equation}
If $v\not =0$ then we can apply Lemma~\ref{FormelNichtlin1}(c) to the weakly compact set $U=\{v_n: n\in \N\}\cup\{v\}$ which does not contain $0$ and find that the expression $\frac{J_1[\|u_n\|_{\mathcal{H}} v_n]}{\|u_n\|_{\mathcal{H}}^2} \to \infty$ as $n\to \infty$. This is not compatible with \eqref{rescale} and hence the weak limit is $v=0$. 

\smallskip

\noindent
{\em Step 2:} Next, let us show that $v_n^+ \to 0$ in $L^{p+1}(\Gamma\times\T_T)$ is impossible.
Since $J_1\geq 0$ we conclude from \eqref{rescale} that $\|v_n^-\|_{\mathcal{H}}^2 \leq \|v_n^+\|_{\mathcal{H}}^2$ which together with $\|v_n^-\|_{\mathcal{H}}^2 + \|v_n^+\|_{\mathcal{H}}^2=1$ implies that $\|v_n^+\|_{\mathcal{H}}^2\geq 1/2$.  Since $v_n$ is a positive multiple of $u_n$ (which itself belongs to $\mathcal{M}$) Lemma~\ref{FormelNichtlin2}(a) together with $\|v_n^+\|_{\mathcal{H}}^2\geq 1/2$ imply that for any $s>0$
\begin{equation} \label{unmoeglich}
J[u_n] \geq J[s v_n^+] = s^2 \|v_n^+\|^2_{\mathcal{H}}- J_1(s v_n^+) \geq \frac{s^2}{2} - \frac{2|s|^{p+1}}{p+1}\|v_n^+\|_{L^{p+1}(\Gamma\times\T_T)}^{p+1}.
\end{equation}
The left hand side is bounded since $(u_n)_{n\in \N}$ is a Palais-Smale sequence. Thus, choosing $s>0$ large, we cannot have $\|v_n^+\|_{L^{p+1}(\Gamma\times\T_T)}^{p+1}\to 0$ as $n\to \infty$ in \eqref{unmoeglich}. 

\smallskip
\noindent
{\em Step 3:} Shifting $v_n^+$. By Step 2, i.e., $v_n^+$ not converging to $0$ in $L^{p+1}(\Gamma\times\T_T)$, Lemma~\ref{ConccompLemma} applies and we find $\delta>0$, a sequence of integers $(m_n)_{n\in\N}$ and a subsequence of $(v_n)_{n\in\N}$ (again denoted by $(v_n)_{n\in\N}$) such that
\begin{align} \label{Beinsyn}
\int_{\Gamma_{m_n}\times\T_T} |v_n^+|^2 \,dx\,dt \geq \delta>0 \text{ for all } n\in\N.
\end{align}
Next we shift $v_n^+$ in such a way that we can make use of compact embeddings for the shifted sequence. To this end we define new functions $v_n^\ast$ by 
\begin{align*}
v_n^{\ast}(x,t) \coloneqq v_n(x +2\pi m_n,t).
\end{align*}
Note that shifting does not change norms in $\mathcal{H}$ and shifting commutes with the spectral projections $P^{\pm}$ since the operators $L_k$ are shift invariant, i.e., $v_n^{\ast,+}=v_n^{+,\ast}$. Then \eqref{Beinsyn} entails
\begin{align*}
\int_{\Gamma_0\times\T_T} |v_n^{\ast,+}|^2 \,dx\,dt = \int_{\Gamma_{m_n}\times\T_T} |v_n^+|^2 \,dx\,t\geq \delta \text{ for all } n\in\N.
\end{align*}
We know that (up to a subsequence) $v_n^{\ast}\rightharpoonup v^{\ast}\in \mathcal{H}$ as $n\to \infty$. The compact embedding into $L^2(\Gamma_0\times\T_T)$ from Theorem~\ref{hlp} yields $\|v^{\ast,+}\|_{L^2(\Gamma\times\T_T)}\neq 0$, i.e., $v^{\ast,+}\not =0$ and hence $v^{\ast} \not=0$. This, however, contradicts the observation $v^{\ast}= w\mbox{-}\lim_{n\to \infty} v_n^\ast=0$ from the beginning of the proof. This contradiction finishes the proof of the boundedness of Palais-Smale sequences of $J|_{\mathcal{M}}$. 
\end{proof}

Finally, we can turn to our overall goal of this section and verify the following statement.

\begin{thm} \label{MinThm}
The functional $J$ admits a ground state, i.e., there exists $u\in \mathcal{M}$ such that $J'[u]=0$ and $J[u]=\inf_{v\in\mathcal{M}} J[v]$.
\end{thm}

The proof requires the following variant of a concentration-compactness Lemma of P.~L.~Lions, cf. Lemma~1.21 in \cite{Willem} for a similar result in Sobolev-spaces. Its proof is given in the Appendix.

\begin{lem} \label{ConccompLemma}
Let $q\in [2,\infty)$ be given and let $(u_n)_{n\in\N}$ be a bounded sequence in $\mathcal{H}$ and 
\begin{align} \label{Bedconccomp}
\sup_{m\in \Z} \int_{\Gamma_m\times\T_T} |u_n|^q \,dx\,dt \to 0 \text{ as } n\to\infty.
\end{align}  
Then $u_n\to 0$ in $L^{\tilde{q}}(\Gamma\times\T_T)$ as $n\to\infty$ for all $\tilde{q}\in (2,\infty)$.
\end{lem}

\noindent
\textit{Proof of Theorem~\ref{MinThm}:} Conditions (B1), (B2) and (i) and (ii) of Theorem~35 in \cite{SW} are fulfilled, and only (iii) does not hold, so that $J$ does not satisfy the Palais-Smale condition. As a consequence, Theorem~35 in \cite{SW} only provides a minimizing Palais-Smale $(u_n)_{n\in \N}$ in $\mathcal{M}$ with 
$J'[u_n]\to 0$ as $n\to\infty$. Lemma~\ref{NichtLin5} guarantees that $(u_n)_{n\in\N}$ is bounded. Thus, there is $u\in\mathcal{H}$, and a subsequence (again denoted by $(u_n)_{n\in\N}$) such that $u_n\rightharpoonup u$ as $n\to\infty$. We now proceed in four steps: 

\smallskip

First claim: $J'[u]=0$. Let $k\in\Zodd$. By Remark~\ref{dense} it is enough to check $J'[u](v)=0$ for $v\in \mathcal{H}$ with compact support. For such $v$ we conclude by weak convergence that 
$$
J_0'[u_n](v)= B(u_n,v) \to B(u,v)=J_0'[u](v) \mbox{ as } n\to \infty
$$
and due to the compact support property of $v$ and the compactness of the embedding $\mathcal{H}\hookrightarrow L^{p+1}(\Gamma_K\times\T)$, $1<p<\infty$, for any compact subgraph $\Gamma_K\subset\Gamma$, cf. Theorem~\ref{hlp}, we obtain
\begin{align*}
J_1'[u_n](v)&= \int_{\Gamma_K\times\T_T} |u_n|^{p-1} u_n v \,dx\,t \to J_1'[u](v) \mbox{ as } n\to \infty.
\end{align*}
Combining the two convergence results we deduce $J'[u]=0$. Note that this chain of arguments only uses that $(u_n)_{n\in\N}$ is a Palais-Smale sequence for $J$ and not $u_n\in \mathcal{M}$.

\smallskip

Second claim: Here we show the existence of a new Palais-Smale sequence $(v_n)_{n\in\N}$ such that $J[v_n]\to \inf_{\mathcal{M}}J$ and that its weak limit $v$ belongs to $\mathcal{M}$ (we do not claim that $v_n\in \mathcal{M}$). For this purpose we can repeat Steps 2 and 3 from the proof of Lemma~\ref{NichtLin5}. First we obtain that $u_n^+$ does not converge to $0$ in $L^{p+1}( \Gamma\times\T_T)$. From this we obtain (via Lemma~\ref{ConccompLemma}) that 
\begin{align} \label{tildeuneqnulleins}
\liminf_{n\to\infty} \sup_{m\in \Z} \int_{\Gamma_m\times\T_T} |u_{n}^+|^2 \,dx\,dt>0.
\end{align}
Therefore we find $\delta>0$, a sequence $(m_n)_{n\in \N}$ in $\Z$ and a subsequence of $(u_n)_{n\in\N}$ (again denoted by $(u_n)_{n\in\N}$) such that 
\begin{align} \label{Beinsyn_again}
\int_{\Gamma_{m_n}\times\T_T} |u^+_n|^2 \,dx\,t \geq \delta>0 \text{ for all } n\in\N.
\end{align}
Setting  
\begin{align*}
v_n(x,t) \coloneqq u_n(x+2\pi m_n,t)
\end{align*}
we obtain that $(v_n)_{n\in \N}$ is again a Palais-Smale sequence for $J$ with $\lim_{n\to\infty} J[v_n]=\inf_{\mathcal{M}} J$ and (as in Step 3 of Lemma~\ref{NichtLin5}) that
\begin{align*}
\int_{\Gamma_0\times\T_T} |v_n^+|^2 \,dx\,dt \geq \delta>0 \text{ for all } n\in\N.
\end{align*}
By making use of the compact embedding to $L^2(\Gamma_0\times\T_T)$ from Theorem~\ref{hlp} up to a subsequence we find that $v_n\rightharpoonup v\in \mathcal{H}$ as $n\to \infty$ with $v\not=0$. The property $J'[v]=0$ follows from the first claim. It remains to show $v^+\neq 0$. Assume by contradiction that $v^+=0$, i.e., $v=v^-$. By testing $J'[v]=0$ with $v$ we infer
\begin{align*}
- \|v^-\|_{\mathcal{H}}^2 = \int_{\Gamma\times\T_T} |v|^{p+1} \,dx\,dt,
\end{align*} 
a contradiction since the two expressions have different signs. Thus, $v\in \mathcal{M}$.

\smallskip

Third claim: $v$ minimizes $J$ on $\mathcal{M}$. Since $v\in \mathcal{M}$ we obviously have $J[v]\geq \inf_{\mathcal{M}} J$. Since for a suitable subsequence $v_n\to v$ pointwise a.e. on $\Gamma\times\T_T$ the reverse inequality is a consequence of Fatou's Lemma as follows:
\begin{align*}
\inf_{\mathcal{M}} J &= \lim_{n\to \infty} J[v_n]- \frac{1}{2} J'[v_n](v_n) = 
\frac{p-1}{p+1}\lim_{n\to \infty} \int_{\Gamma\times\T_T}|v_n|^{p+1} \,dx\,dt \\
& \geq \frac{p-1}{p+1}\int_{\Gamma\times\T_T} |v|^{p+1} \,dx\,dt = J[v]-\frac{1}{2} J'[v](v) = J[v].
\end{align*}

\smallskip

Fourth claim (infinitely many critical points): It remains to show the multiplicity result. Recall that Lemma~\ref{specific_values} and hence Theorem~\ref{hlp} guarantees the embedding of $\mathcal{H}$ into all $L^q(\Gamma\times\T_T)$-spaces for $2\leq q<\infty$ provided $\kappa\in \Nodd$ is chosen sufficiently large. Here $\kappa\Zodd$ selects the admissible frequencies $k\omega$ with $k\in \kappa\Zodd$ in the Fourier expansion in time of elements in $\mathcal{H}$. These elements are then $\frac{2\pi}{\kappa}$-antiperiodic in time. If we consider our construction of a minimizer of $J|_{\mathcal{M}}$ for all sufficiently large $\kappa\in \Nodd$, then we obtain a sequence of minimizers $(u_{\kappa})_{\kappa\in\Nodd}$, and in particular $u_\kappa\not =0$. Some of the elements in this sequence may be equal, but we will see that infinitely many are mutually different. For this purpose, assume for contradiction that the set of minimizers $\{u_{\kappa}: \kappa\in \Nodd \text{ sufficiently large}\}$ is finite. But since $u_\kappa\rightharpoonup 0$ as $\kappa\to\infty$ and none of the $u_\kappa$ is equal to $0$ this is impossible. This contradiction shows the existence of infinitely many distinct critical points of the function $J$ and finishes the proof of the theorem.
\qed

\begin{rem} Let us explain how the case of ''$-$'' in \eqref{waveeq} can be treated. In this case one keeps the functional $J_1$ but replaces $J_0$ by $-J_0$ and flips the spaces $\mathcal{H}^+$ and $\mathcal{H}^-$. Since $J_0$ is an indefinite functional this is without relevance for the proof strategy. All proofs of this section can be carried over with no change.
\end{rem}

\section{Regularity part of Theorem~\ref{main} and Theorem~\ref{prop_ivp}} \label{sec:regularity}

Here we first prove the regularity part of Theorem~\ref{main}. The claimed regularity follows from Lemma~\ref{nonlinear_reg1}, where we show that every critical point $u\in \mathcal{H}$ of the functional $J$ is a weak solution $u\in H_c^1(\Gamma\times\T_T)$ of \eqref{waveeq}, and from Lemma~\ref{nonlinear_reg2}, where we show that weak solutions belong to $H^2(\Gamma\times\T_T)$ and are strong solutions which fulfill \eqref{waveeq} pointwise a.e. as well as the Kirchhoff conditions \eqref{kirchhoff_abl}. Finally, we show Theorem~\ref{prop_ivp}.

We begin by first analyzing solutions of the linear problem
\begin{equation} \label{waveeq_linear}
Lv = f \qquad \mbox{ with } L = \partial_t^2-\partial_x^2.
\end{equation}
Here we let $f\in \mathcal{H}'$, set $v(x,t)=\sum_{k\in \kappa\Zodd} v_k(x)e^{ik\omega t}$ and consider a solution $v\in\mathcal{H}$ of \eqref{waveeq_linear} in the sense that 
$$
\sum_{k\in \kappa\Zodd} b_{L_k}(v_k,\phi_k) = \langle f,\phi\rangle \quad \mbox{ for all } \phi \in \mathcal{H}
$$
where $\langle \cdot,\cdot\rangle$ is the duality bracket between $\mathcal{H}'$ and $\mathcal{H}$. 

\begin{lem} \label{lin_reg1} For $f\in \mathcal{H}'$ there exists a unique solution $v\in \mathcal{H}$ of \eqref{waveeq_linear} with $\|v\|_{\mathcal H}= \|f\|_{{\mathcal H}'}$. If additionally $f\in L^2(\Gamma\times\T_T)_{\text{symm}}$ then $v\in H^1_c(\Gamma\times\T_T)_{\text{symm}}$, $\|v\|_{H^1(\Gamma\times\T_T)}\leq C\|f\|_{L^2(\Gamma\times\T_T)}$ and $v$ satisfies \eqref{waveeq_linear} in the sense that 
$$
\int_0^T \int_\Gamma \left(-v_t\phi_t + v_x\phi_x+\alpha v\phi\right) \,d(x,t) = \int_0^T \int_\Gamma f \phi \,d(x,t) \quad \mbox{ for all } \phi \in \mathcal{H}\cap H^1_c(\Gamma\times\T_T)_{\text{symm}}.
$$
\end{lem}

\begin{proof} The splitting $\mathcal{H}=\mathcal{H}^+\oplus \mathcal{H}^-$ transfers to $\mathcal{H}'=(\mathcal{H}^+)'\oplus(\mathcal{H}^-)'$ and hence $f\in \mathcal{H}'$ can be split as $f=f^++f^-$ with $f^{\pm}\in (\mathcal{H}^\pm)'$. If we denote by $\langle \cdot,\cdot\rangle_\mathcal{H}$ the inner product in $\mathcal{H}$ then by the Riesz representation theorem there exist $v^\pm\in \mathcal{H}^\pm$ with $\langle v^\pm,\phi\rangle_{\mathcal{H}} = \pm\langle f^\pm,\phi\rangle$ for all $\phi\in \mathcal{H}$. Hence 
$$
\sum_{k\in \kappa\Zodd} b_{L_k}(v_k,\phi_k) = \langle v^+,\phi^+\rangle_{\mathcal H} - \langle v^-,\phi^-\rangle_{\mathcal H} = \langle f^+,\phi^+\rangle + \langle f^-,\phi^-\rangle =\langle f,\phi\rangle
$$
for all $\phi\in \mathcal{H}$. This finishes the existence part of the lemma. If we know additionally $f\in L^2(\Gamma\times\T_T)_{\text{symm}}$ then using the Floquet-Bloch-Fourier expansion we have  
$$
f(x,t) = \int_{-\frac{1}{2}}^\frac{1}{2} \sum_{k\in\kappa\Zodd}\sum_{m\in\Z} (\tilde f_k)_m(l) \phi_m(l,x) e^{i(lx+\omega kt)}\,dl
$$
and hence 
\begin{align}
v(x,t) & = \int_{-\frac{1}{2}}^\frac{1}{2} \sum_{k\in\kappa\Zodd}\sum_{m\in\Z} \frac{(\tilde f_k)_m(l)}{\lambda_m(l)-k^2\omega^2+\alpha} \phi_m(l,x) e^{i(lx+\omega kt)}\,dl \label{darstellung}\\
& = v_1(x,t) + v_2(x,t) \nonumber
\end{align}
where 
\begin{align*} 
v_1(x,t) & =\int_{-\frac{1}{2}}^\frac{1}{2} \sum_{\substack{k\in\kappa\Zodd\\ |k|\leq K}}\sum_{m\in\Z} \frac{(\tilde f_k)_m(l)}{\lambda_m(l)-k^2\omega^2+\alpha} \phi_m(l,x) e^{i(lx+\omega kt)}\,dl \\
v_2(x,t) & =\int_{-\frac{1}{2}}^\frac{1}{2} \sum_{\substack{k\in\kappa\Zodd\\ |k|>K}}\sum_{m\in\Z} \frac{(\tilde f_k)_m(l)}{\lambda_m(l)-k^2\omega^2+\alpha} \phi_m(l,x) e^{i(lx+\omega kt)}\,dl
\end{align*}
and where $K$ is chosen so large that $k^2\omega^2-\alpha \geq \frac{k^2\omega^2}{2}$ for $|k|>K$. Clearly, by \eqref{spec} we have $|\lambda_m(l)-k^2\omega^2+\alpha|\geq \delta^*$. Additionally, if $|k|\leq K$ and $|m|> m_0(K)$ then $|\lambda_m(l)-k^2\omega^2+\alpha|\geq \delta_0(\lambda_m(l)+1)$. For $|k|\leq K$ and $|m|\leq m_0(K)$ then the same holds by \eqref{spec} by possibly diminishing $\delta_0$. Hence for all $|k|\geq K$ and all $m\in\Z$ we have $|\lambda_m(l)-k^2\omega^2+\alpha|^2\geq \delta\delta_0(\lambda_m(l)+1)$. Thus 
\begin{align*}
\|\partial_t v_1\|_{L^2(\Gamma\times\T_T)}^2 &= T\int_{-\frac{1}{2}}^\frac{1}{2} \sum_{\substack{k\in\kappa\Zodd\\ |k|\leq K}}\sum_{m\in\Z} \frac{|(\tilde f_k)_m(l)|^2}{|\lambda_m(l)-k^2\omega^2+\alpha|^2} \,dl \\
& \leq \frac{K^2\omega^2}{{\delta^\ast}^2} T\int_{-\frac{1}{2}}^\frac{1}{2} \sum_{\substack{k\in\kappa\Zodd\\ |k|\leq K}}\sum_{m\in\Z} |(\tilde f_k)_m(l)|^2 \,dl
\end{align*}
and 
\begin{align*}
\|\partial_x v_1\|_{L^2(\Gamma\times\T_T)}^2 &=T\int_{-\frac{1}{2}}^\frac{1}{2} \sum_{\substack{k\in\kappa\Zodd\\ |k|\leq K}}\sum_{m\in\Z} \frac{|(\tilde f_k)_m(l)|^2 \lambda_m(l)}{|\lambda_m(l)-k^2\omega^2+\alpha|^2} \,dl \\
&\leq \frac{T}{\delta\delta_0}\int_{-\frac{1}{2}}^\frac{1}{2} \sum_{\substack{k\in\kappa\Zodd\\ |k|\leq K}}\sum_{m\in\Z} |(\tilde f_k)_m(l)|^2\,dl.
\end{align*}
In a similar way for $v_2$, we can use the estimate \eqref{sufficient_condition} of Lemma~\ref{sufficient_for_embedding} for the lower bound $|\lambda_m(l)-k^2\omega^2+\alpha|^2 \geq \delta^2(\sqrt{\lambda_m(l)}+ \sqrt{k^2\omega^2-\alpha})^2$ so that
\begin{align*}
\|\partial_t v_2\|_{L^2(\Gamma\times\T_T)}^2 & \leq \frac{T}{\delta^2}\int_{-\frac{1}{2}}^\frac{1}{2} \sum_{\substack{k\in\kappa\Zodd\\ |k|> K}}\sum_{m\in\Z} \frac{|(\tilde f_k)_m(l)|^2 \omega^2 k^2}{(\sqrt{\lambda_m(l)}+ \sqrt{k^2\omega^2-\alpha})^2} \,dl \\
& \leq \frac{2T}{\delta^2}\int_{-\frac{1}{2}}^\frac{1}{2} \sum_{\substack{k\in\kappa\Zodd\\ |k|> K}}\sum_{m\in\Z} |(\tilde f_k)_m(l)|^2\,dl
\end{align*}
and 
\begin{align*}
\|\partial_x v_2\|_{L^2(\Gamma\times\T_T)}^2 & \leq \frac{T}{\delta^2}\int_{-\frac{1}{2}}^\frac{1}{2} \sum_{\substack{k\in\kappa\Zodd\\ |k|> K}}\sum_{m\in\Z} \frac{|(\tilde f_k)_m(l)|^2 \lambda_m(l)}{(\sqrt{\lambda_m(l)}+ \sqrt{k^2\omega^2-\alpha})^2} \,dl \\
& \leq \frac{T}{\delta^2} \int_{-\frac{1}{2}}^\frac{1}{2} \sum_{\substack{k\in\kappa\Zodd\\ |k|> K}}\sum_{m\in\Z} |(\tilde f_k)_m(l)|^2\,dl.
\end{align*}
so that altogether we have $\|\partial_t v\|_{L^2(\Gamma\times\T_T)}^2 \leq (\frac{K^2\omega^2}{{\delta^\ast}^2} + \frac{2}{\delta^2}) \|f\|_{L^2(\Gamma\times\T_T)}^2$ and $\|\partial_x v\|_{L^2(\Gamma\times\T_T)}^2 \leq \frac{2}{\delta^2} \|f\|_{L^2(\Gamma\times\T_T)}^2$. Our calculation also shows that $v$ given by \eqref{darstellung} is an $H^1(\Gamma_N\times\T_T)$-Bochner-integral for every compact subgraph $\Gamma_N$. This means that for a vertex $P$ with adjoining edges $\Gamma_0, \Gamma_+, \Gamma_-$ we know that $\trace{v(\cdot,t)}|_{x=P_0}$, $\trace{v(\cdot,t)}|_{x=P_+}$ and $\trace{v(\cdot,t)}|_{x=P_-}$ exist. Since for any of theses traces we have (due to the properties of the Bochner integral)
$$
\trace{v(\cdot,t)}|_{x=P} = \int_{-\frac{1}{2}}^\frac{1}{2} \sum_{k\in\kappa\Zodd}\sum_{m\in\Z} \frac{(\tilde f_k)_m(l)}{\lambda_m(l)-k^2\omega^2+\alpha} \trace{\bigl(\phi_m(l,\cdot) e^{i(l\cdot+\omega kt)}\bigr)}|_{x=P}\,dl
$$
the equality of the three traces at the vertex $P$ is passed on from $\phi_m(l,\cdot)e^{i(l\cdot+\omega kt)}$ to the solution $v$. Hence $v\in H^1_c(\Gamma\times\T_T)_{\text{symm}}$, and therefore it is easy to verify that 
$$
\sum_{k\in \kappa\Zodd} b_{L_k}(v_k,\phi_k) = 
\int_0^T \int_\Gamma \left(-v_t\phi_t + v_x\phi_x+\alpha v\phi\right) \,d(x,t) \mbox{ for all } \phi\in \mathcal{H}\cap H^1_c(\Gamma\times \T_T)_{\text{symm}}.
$$
This finishes the claim.
\end{proof}

The next lemma transfers the linear result to the nonlinear regime.

\begin{lem} \label{nonlinear_reg1} Let $u\in \mathcal{H}$ be a critical point of the functional $J$. Then $u\in H^1_c(\Gamma\times\T_T)_{\text{symm}}$ and $u$ satisfies \eqref{waveeq} in the sense of Definition~\ref{def_weak_sol}.
\end{lem}

\begin{proof} The proof consists in applying Lemma~\ref{lin_reg1} to $u$ with $f=\pm |u|^{p-1}u$. Note that $f\in L^2(\Gamma\times\T_T)_{\text{symm}}$ because $u\in L^q(\Gamma\times\T_T)_{\text{symm}}$ for all $q\in [2,\infty)$ by Lemma~\ref{specific_values}. Any time periodic test function $\phi\in H^1_c(\Gamma\times T_\T)$ can be split into a symmetric part belonging to $H^1_c(\Gamma\times T_\T)_\text{symm}$ and a corresponding antisymmetric part. Whereas testing with the symmetric part is covered by Lemma~\ref{lin_reg1}, testing with the antisymmetric part trivially yields $0$ in every integral.
\end{proof}

In order to reach higher time regularity we apply the method of differences and define for $u\in \mathcal{H}$ and $h\not = 0$ the function
$$
D_h u(x,t) := \frac{u(x,t+h)-u(x,t)}{h}.
$$
Note that $D_h u\in {\mathcal H}$. 

\begin{lem} \label{differenz} Let $v\in \mathcal{H}$. 
\begin{itemize}
\item[(i)] If $v_t\in L^2(\Gamma\times\T_T)$ then $\|D_h v\|_{L^2(\Gamma\times\T_T)} \leq \|\partial_t v\|_{L^2(\Gamma\times\T_T)}$.
\item[(ii)] If $\sup_{h>0}\|D_h v\|_{L^2(\Gamma\times\T_T)}<\infty$ then $\partial_t v\in L^2(\Gamma\times\T_T)$. 
\end{itemize}
In both cases $D_hv \to \partial_t v$ in $L^2(\Gamma\times\T_T)$ as $h\to 0$.
\end{lem}

\begin{proof} The estimate in (i) follows from $|e^{is}-1|=2|\sin\frac{s}{2}|\leq |s|$ and thus
\begin{align*}
\|D_h v\|_{L^2(\Gamma\times\T_T)}^2 & = \sum_{k\in \kappa\Zodd} \|v_k\|_{L^2(\Gamma)}^2 \int_0^T \left|\frac{e^{ikw(t+h)}-e^{ikwt}}{h}\right|^2\,dt \\
& \leq \sum_{k\in \kappa\Zodd} \|v_k\|_{L^2(\Gamma)}^2 Tk^2\omega^2 = \|\partial_t v\|_{L^2(\Gamma\times\T_T)}^2.
\end{align*}
Moreover, by dominated convergence we find that
\begin{align*} 
\|D_h v-\partial_t v\|_{L^2(\Gamma\times\T_T)}^2 & = \sum_{k\in \kappa\Zodd} \|v_k\|_{L^2(\Gamma)}^2 \int_0^T \left|\frac{e^{ik\omega(t+h)}-e^{ik\omega t}}{h}-ik\omega e^{ik\omega t}\right|^2\,dt \\
&= \sum_{k\in \kappa\Zodd} \|v_k\|_{L^2(\Gamma)}^2 \int_0^T \left|\frac{e^{ik\omega h}-1}{h}-ik\omega \right|^2\,dt \to 0
\end{align*}
as $h\to 0$. In case (ii) there is a weakly convergent subsequence $D_{h_n} v\rightharpoonup w$ in $L^2(\Gamma\times\T_T)$ as $n\to \infty$. Starting from the identity for $\phi\in H^1_c(\Gamma\times\T_T)$
$$
\int_\Gamma \int_0^T D_{h_n} v \phi \,dt\,dx = -\int_\Gamma \int_0^T v D_{-h_n} \phi \,dt\,dx
$$
and taking the limit $n\to \infty$ we get 
$$
\int_\Gamma \int_0^T w \phi \,dt\,dx = -\int_\Gamma \int_0^T v \partial_t \phi \,dt\,dx,
$$
i.e., $\partial_t v$ exists and $=w\in L^2(\Gamma\times\T_T)$. The convergence result then follows from (i).
\end{proof}

The previous lemma will now be applied to weak solutions of the nonlinear wave equation \eqref{waveeq} in the following way.

\begin{lem} Let $u\in\mathcal{H}$ be a weak solution of \eqref{waveeq} and let $w:= |u|^{p-1} u$. Then 
\begin{align*} 
& \sup_{h>0} \|D_h w\|_{\mathcal{H}'}<\infty,  && \sup_{h>0} \|D_h w\|_{L^2(\Gamma\times\T_T)}<\infty,\\
& \sup_{h>0} \|D_h u\|_{\mathcal H}<\infty, &&  \sup_{h>0} \|\partial_t D_h u\|_{L^2(\Gamma\times\T_T)}<\infty,\\
& \sup_{h>0} \|\partial_x D_h u\|_{L^2(\Gamma\times\T_T)}<\infty
\end{align*}
which implies $\partial_t^2,\partial_x\partial_t u\in L^2(\Gamma\times\T_T)$.
\end{lem}

\begin{proof} Note that $D_h w(x,t) = g_h(x,t) D_h u(x,t)$ where 
$$
g_h(x,t) = \frac{|u(x,t+h)|^{p-1} u(x,t+h)- |u(x,t)|^{p-1}u(x,t)}{u(x,t+h)-u(x,t)}.
$$
By using the inequality $\bigl||s|^{p-1}s-|t|^{p-1}t\bigr|\leq p|\xi|^{p-1}|s-t|\leq p(|s|^{p-1}+|t|^{p-1})|s-t|$ for $s,t\in\R$ and $\xi$ between $s$ and $t$, we see that
$$
|g_h(x,t)| \leq p( |u(x,t+h)|^{p-1} + |u(x,t)|^{p-1}).
$$
Moreover, $D_h u\in\mathcal{H}\cap H^1_c(\Gamma\times\T_T)$ is a weak solution of $L D_h u = \pm D_h w$ in the sense of Lemma~\ref{lin_reg1}. Let $r>2$ be such that $(p-1)r \geq 2$. Then with $q:=(\frac{1}{r}+\frac{1}{2})^{-1}\in (1,2)$ and Lemma~\ref{differenz}(i) we find 
\begin{align*}
\| D_h w\|_{L^q(\Gamma\times\T_T)} & \leq \|g_h\|_{L^r(\Gamma\times\T_T)} \|D_h u\|_{L^2(\Gamma\times\T_T)} \\
& \leq 2p\|u\|_{L^{(p-1)r}(\Gamma\times\T_T)}^{p-1} \|D_h u\|_{L^2(\Gamma\times\T_T)} \\
& \leq 2p\|u\|_{L^{(p-1)r}(\Gamma\times\T_T)}^{p-1} \|\partial_t u\|_{L^2(\Gamma\times\T_T)}
\end{align*}
and hence $\sup_{h>0} \|D_h w\|_{\mathcal{H}'}<\infty$. With the help of the first statement in Lemma~\ref{lin_reg1} we see that $\sup_{h>0} \|D_h u\|_{\mathcal H}= \sup_{h>0} \|D_h w\|_{\mathcal{H}'}<\infty$. Now we apply H\"older's inequality one more time with $\frac{1}{t}+\frac{1}{t'}=1$, $t>1$ and get 
\begin{align*}
\| D_h w\|_{L^2(\Gamma\times\T_T)} & \leq \|g_h\|_{L^{2t'}(\Gamma\times\T_T)} \|D_h u\|_{L^{2t}(\Gamma\times\T_T)} \\
& \leq 2p\|u\|_{L^{(p-1)2t'}(\Gamma\times\T_T)}^{p-1} \|D_h u\|_{L^{2t}(\Gamma\times\T_T)}.
\end{align*}
In case $1<p<2$ we choose $t'=\frac{1}{p-1}\in (1,\infty)$ and $t=\frac{1}{2-p}\in (1,\infty)$ and get 
\begin{align*}
\sup_{h>0} \| D_h w\|_{L^2(\Gamma\times\T_T)} & \leq 2p\|u\|_{L^2(\Gamma\times\T_T)}^{p-1} \sup_{h>0} \|D_h u\|_{L^{2t}(\Gamma\times\T_T)}\\
& \leq 2pC\|u\|_{L^2(\Gamma\times\T_T)}^{p-1} \sup_{h>0}\|D_h u\|_{\mathcal{H}}<\infty.
\end{align*}
In the case $p\geq 2$ we use $t,t'>1$, $(p-1)2t'>(p-1)2\geq 2$ to get 
\begin{align*}
\sup_{h>0} \| D_h w\|_{L^2(\Gamma\times\T_T)} & \leq 2p\|u\|_{L^{(p-1)2t'}(\Gamma\times\T_T)}^{p-1} \sup_{h>0} \|D_h u\|_{L^{2t}(\Gamma\times\T_T)}\\
& \leq 2pC\|u\|_{\mathcal{H}}^{p-1} \sup_{h>0}\|D_h u\|_{\mathcal{H}}<\infty.
\end{align*}
In both cases we have established $\sup_{h>0} \|D_h w\|_{L^2(\Gamma\times\T_T)}<\infty$. Since $D_hu$ is a solution of $LD_h u =\pm D_h w$ we can use Lemma~\ref{lin_reg1} and get $\sup_{h>0} \|\partial_x D_h u\|_{L^2(\Gamma\times\T_T)}<\infty$ as well as $\sup_{h>0} \|\partial_t D_h u\|_{L^2(\Gamma\times\T_T)}<\infty$. And since $D_h$ commutes with $\partial_x, \partial_t$ we get by Lemma~\ref{differenz} that $\partial_t^2 u, \partial_x\partial_t u \in L^2(\Gamma\times\T_T)$ as claimed. 
\end{proof}

The final piece of the regularity statement of Theorem~\ref{main} is given in the next lemma.

\begin{lem} \label{nonlinear_reg2} Let $u\in\mathcal{H}$ be a weak solution of \eqref{waveeq}. Then in addition to $\partial_x u, \partial_t u, \partial_t^2 u, \partial_x\partial_t u$ also $\partial_x^2 u \in L^2(\Gamma\times\T_T)$. Moreover, the solutions satisfy \eqref{waveeq} pointwise a.e. on $\Gamma\times\T_T$ and the Kirchhoff conditions \eqref{kirchhoff_abl} holds at the vertices for almost all $t\in\T_T$.
\end{lem}

\begin{proof} First we note that with $w:= |u|^{p-1}u$ we have that $b_{L_k}(u_k,\phi_k) = \int_\Gamma \pm w_k \phi_k\,dx$ for all $k\in \kappa\Zodd$ and all $\phi\in\mathcal{H}$. In particular 
$$
\int_\Gamma u_k' \phi_k' \,dx = \int_\Gamma \left((\omega^2 k^2-\alpha)u_k \pm w_k\right)\phi_k \mbox{ for all } \phi\in \mathcal{H}
$$
which shows that $u_k\in \mathcal{D}(\Delta)= \{v\in H^2(\Gamma)\cap H^1_c(\Gamma)_{\text{symm}}: v \mbox{ satisfies } \eqref{kirchhoff_abl}\}$. Moreover 
$$
\int_{\Gamma\times\T_T} \partial_x u\partial_x \phi \,d(x,t) = \int_{\Gamma\times\T_T}(\underbrace{\pm |u|^{p-1}u -\alpha u -\partial_t^2 u}_{\in L^2(\Gamma\times\T_T)})\phi\,d(x,t) 
$$
for all $\phi\in \mathcal{H}\cap H^1_c(\Gamma\times\T_T)$ implies that $\partial_x^2 u\in L^2(\Gamma\times\T_T)$ and that 
\begin{align*}
\sum_{\substack{n \\ j\in\{\pm\}}} &\int_0^T \int_{\pi(2n+1)}^{2\pi (n+1)} -\partial_x^2 u^j\phi^j\,d(x,t) +\sum_{\substack{n \\ j\in\{\pm\}}} \int_0^T \partial_x u^j \phi^j\Big|_{\pi(2n+1)}^{2\pi(n+1)}\,dt \\
+ \sum_n & \int_0^T \int_{2\pi n}^{\pi(2n+1)} -\partial_x^2 u^0\phi^0\,d(x,t) +\sum_n \int_0^T \partial_x u^0 \phi^0\Big|_{2\pi n}^{\pi(2n+1)}\,dt \\
& = \int_{\Gamma\times\T_T} \partial_x u \partial_t \phi \,d(x,t) = \int_{\Gamma\times\T_T}(\pm |u|^{p-1}u -\alpha u -\partial_t^2 u)\phi\,d(x,t). 
\end{align*}
From this we deduce that \eqref{waveeq} holds pointwise a.e. in $\Gamma\times\T_T$ as well as the Kirchhoff conditions
\begin{align*}
\partial_x u^+(2\pi n-,t)+ \partial_x u^-(2\pi n-,t) &= \partial_x u^0(2\pi n+,t), \\
\partial_x u^+(\pi(2n+1)+,t)+\partial_x u^-(\pi(2n+1)+,t) &= \partial_x u^0(\pi(2n+1)-,t)
\end{align*}
for all $n\in \Z$ and almost all $t\in\T_T$ as claimed, cf. \eqref{kirchhoff_abl}. This finishes the proof of the lemma.
\end{proof}

Now we finish the proof of Theorem~\ref{prop_ivp}.

\begin{proof}[Proof of Theorem~\ref{prop_ivp}.]
By the regularity properties of Theorem~\ref{main} we already know that $u, \partial_t u, \partial_t^2 u, \partial_x u, \partial_x^2 u, \partial_x\partial_t u \in L^2(\Gamma\times\T_T)$ and hence $u\in L^2(\T_T; H_c^1(\Gamma)\cap H^2(\Gamma))$, $\partial_t^2 u \in L^2(\T_T;L^2(\Gamma))$ as claimed. Moreover, using Definition~\ref{def_weak_sol} with $\phi(x,t)= \varphi(x)(e^{ik\omega t}+ e^{-ik\omega t})$, $\varphi\in H^1_c(\Gamma)$ and setting $w(x,t)= |u(x,t)|^{p-1}u(x,t)$ we get 
$$
\int_\Gamma -k^2 \omega^2 u_k\varphi + u_k' \varphi' +\alpha u_k \varphi \,dx = \pm \int_\Gamma w_k\varphi \,dx
$$
Multiplication with $e^{ik\omega t}$ and summation over $k\in \kappa\Zodd$ implies
\begin{equation} \label{after_testing}
\sum_{k\in\kappa\Zodd} \int_\Gamma (-k^2 \omega^2 u_k e^{ik\omega t} \varphi + u_k' e^{ik\omega t} \varphi' +\alpha u_k e^{ik\omega t} \varphi) \,dx = \pm \sum_{k\in\kappa\Zodd} \int_\Gamma w_k e^{ik\omega t}\varphi \,dx.
\end{equation}
We may use that in the sense of $L^2(\Gamma\times\T_T)$-convergence we have the identities 
\begin{align*}
& \sum_{k\in\kappa\Zodd} -k^2 \omega^2 u_k(x) e^{ik\omega t} = \partial_t^2 u(x,t), && \sum_{k\in\kappa\Zodd} u_k'(x) e^{ik\omega t} = \partial_x u(x,t),\\
& \sum_{k\in\kappa\Zodd} \alpha u_k(x) e^{ik\omega t} = \alpha u(x,t), && \sum_{k\in\kappa\Zodd} w_k(x) e^{ik\omega t} = w(x,t)= |u(x,t)|^{p-1}u(x,t)
\end{align*}
which together with \eqref{after_testing} implies the claimed identity \eqref{ivp}. Also, since $u\in H^1_c(\Gamma\times\T_T)$ we have $u(0,\cdot)=\trace u(t,\cdot)|_{t=0}$. Likewise, since $\partial_t^2 u, \partial_t\partial_x u\in L^2(\Gamma\times\T_T)$ we have $\partial_t u \in H^1_c(\Gamma\times\T_T)$ so that also $\partial_t u(0,\cdot)=\trace \partial_t u(t,\cdot)|_{t=0}$. Using that $\sum_{k\in\kappa\Zodd} \int_\Gamma  u_k''(x) e^{ik\omega t}\,dx = \partial_x^2 u(x,t)$ in the $L^2$-sense we may integrate edge-wise by parts with respect to $x$ in \eqref{after_testing} and obtain (as in the proof of Lemma~\ref{nonlinear_reg2}) the Kirchhoff conditions \eqref{kirchhoff_abl}. 
\end{proof}

\section*{Appendix}

Here we give the proof of the concentration compactness result of Lemma~\ref{ConccompLemma}. The proof differs rather largely from the usual proof in the case where Sobolev spaces are involved. Sobolev spaces with Hilbert space structure have a locality property: functions with disjoint support are orthogonal. A similar statement, i.e., $\|\sum_k \phi_k\|_{W^{m,p}}^p = \sum_k \|\phi_k\|_{W^{m,p}}^p$ if $\supp \phi_j\cap\supp \phi_i=\emptyset$ for $i\not =j$, also holds for Sobolev spaces $W^{m,p}$, $m\in \N$, $p\geq 1$ with only a Banach space structure. The locality property is used heavily in the proof of the concentration compactness principle - but our Hilbert space $\mathcal{H}$ does not seem to have the locality property since it is based on the nonlocal Bloch transform. Therefore we introduce a new Hilbert space $H$, which has the advantage of having the locality property but the disadvantage of embedding only in $L^q(\Gamma\times\T_T)$ for $2\leq q<4$. Together with the better embedding properties of $\mathcal{H}$ this is, however, sufficient for the proof of the concentration compactness principle.

\begin{definition} Define the Hilbert space $H$ by 
\begin{equation}
H=\{u(x,t)=\sum_{k\in\kappa\Zodd} u_k(x) e^{ik\omega t}: u_k\in H^1_c(\Gamma)_{\text{symm}}, \overline{u_k}=u_{-k} \forall k\in\kappa\Zodd \mbox{ with } \|u\|_H<\infty\}
\end{equation}
and where the norm is given by
$$
\|u\|_H^2 = \sum_{k\in\kappa\Zodd} \frac{1}{|k|}\|u_k'\|_{L^2(\Gamma)}^2+ |k|\|u_k\|_{L^2(\Gamma)}^2.
$$
\end{definition}

\begin{lem} \label{embedding_suboptimal} The Hilbert space $H$ has the following properties: $\mathcal{H}$ embeds continuously into $H$ and $H$ embeds continuously in $L^q(\Gamma\times\T_T)$ for $2\leq q<4$. 
\end{lem}

\begin{proof} The first embedding statement follows once we have verified the two inequalities
\begin{equation} \label{bilinear_form}
b_{|L_k|}(v,v) \geq C |k| \|v\|_{L^2(\Gamma)}^2, \quad b_{|L_k|}(v,v) \geq \frac{C}{|k|} \|v'\|_{L^2(\Gamma)}^2 
\end{equation}
for all $v\in H^1_c(\Gamma)_{\text{symm}}$ and a constant $C>0$ which is independent on $v$. For the first inequality we use \eqref{sufficient_condition} from Lemma~\ref{sufficient_for_embedding} and get
$$
|\lambda_m(l)-\omega^2 k^2 +\alpha| = \left|\sqrt{\lambda_m(l)}-\sqrt{\omega^2 k^2-\alpha}\right|\cdot\left|\sqrt{\lambda_m(l)}+\sqrt{\omega^2 k^2-\alpha}\right| \geq \delta C|k|
$$
if $|k|$ is sufficiently large. Then, the first inequality in \eqref{bilinear_form} follows from this estimate using the expression of $b_{|L_k|}$ in terms of the Bloch variables. For the remaining small values of $|k|$ the first inequality in \eqref{bilinear_form} follows from \eqref{spec}. For the second inequality in \eqref{bilinear_form} let us first consider small values of $|k|$. Then the second inequality in \eqref{bilinear_form} is just a consequence of the fact that due to \eqref{spec} the bilinear form $b_{|L_k|}$ generates a norm which is equivalent to the $H^1_c$-norm of the domain of $b_{|L_k|}$. Now we consider the second inequality in \eqref{bilinear_form} for large values of $|k|$. We first take $v\in H^1_c(\Gamma)_{\text{symm}}^+$. Then, by the first inequality in \eqref{bilinear_form}, we have
$$
\int_\Gamma |v'|^2-(\omega^2 k^2-\alpha)|v|^2\,dx \geq c|k|\int_\Gamma |v|^2\,dx \geq \frac{\tilde C}{|k|-\tilde C}(\omega^2 k^2-\alpha) \int_\Gamma|v|^2\,dx
$$
for $\tilde C>0$ small enough and $|k|$ sufficiently large. Hence we find for $|k|$ large 
$$
\int_\Gamma \frac{|k|-\tilde C}{|k|} |v'|^2 - (\omega^2 k^2-\alpha)|v|^2\,dx \geq 0 
$$
from which the second inequality in \eqref{bilinear_form} follows.  Next we take $v\in H^1_c(\Gamma)_{\text{symm}}^-$. In this case $\int_\Gamma |v'|^2-(\omega^2k^2-\alpha)|v|^2\,dx\leq 0$ and hence $\|v'\|_{L^2(\Gamma)} \leq \omega |k| \|v\|_{L^2(\Gamma)}$. Thus, in this case, the second inequality in \eqref{bilinear_form} directly follows from the first. 

Now we turn to the embedding of $H$ into $L^q(\Gamma\times\T_T)$ spaces. For $q\in [2,\infty]$ and $q'=\frac{q}{q-1}$ being the conjugate exponent, let us define the space 
$$
X_q = \{ u=\sum_{k\in\kappa\Zodd} u_k e^{ik\omega t}: \|u\|_{X_q}<\infty\}
$$
with the norm
$$
\|u\|_{X_q} := \left(\sum_{k\in\kappa\Zodd} \|u_k\|_{L^q(\Gamma)}^{q'}\right)^{1/q'}.
$$
To see the relation between $L^q(\Gamma\times\T)$ and $X_q$ we observe that $\|u\|_{L^2(\Gamma\times\T_T)}=\|u\|_{X_2}$ and $\|u\|_{L^\infty(\Gamma\times\T_T)}\leq \|u\|_{X_\infty}$. Hence the Riesz-Thorin Theorem implies that $X_q$ embeds into $L^q(\Gamma\times\T_T)$ for all $2\leq q\leq \infty$. To complete our embedding statement for $H$ we need to see that it embeds into $X_q$ for $2\leq q<4$. This will be done next. For $q=2$ the embedding is clear. For $2<q<4$ we use the Gagliardo-Nirenberg inequality
\begin{equation} \label{gn}
\|v\|_{L^q(\Gamma)} \leq C_{GN} \|v'\|_{L^2(\Gamma)}^\theta \|v\|_{L^2(\Gamma)}^{1-\theta}, \qquad v\in H^1_c(\Gamma)_{\text{symm}}
\end{equation}
with $\theta=\frac{1}{2}-\frac{1}{q}$. Next we use \eqref{gn} and estimate
\begin{align*}
\sum_{k\in\kappa\Zodd} \|u_k\|_{L^q(\Gamma)}^{q'} & \leq C_{GN}^{q'} \sum_{k\in\kappa\Zodd} \|u_k'\|_{L^2(\Gamma)}^{q'(\frac{1}{2}-\frac{1}{q})} \|u_k\|_{L^2(\Gamma)}^{q'(\frac{1}{2}+\frac{1}{q})} \\
& = C_{GN}^{q'} \sum_{k\in\kappa\Zodd} (|k|^{-\frac{1}{2}}\|u_k'\|_{L^2(\Gamma)})^{q'(\frac{1}{2}-\frac{1}{q})} (|k|^\frac{1}{2} \|u_k\|_{L^2(\Gamma)})^{q'(\frac{1}{2}+\frac{1}{q})} |k|^\frac{-1}{q-1}.
\end{align*}
A triple H\"older inequality with exponents $\frac{4(q-1)}{q-2}$, $\frac{4(q-1)}{q+2}$, and $\frac{2(q-1)}{q-2}$ leads to 
\begin{eqnarray*}
\lefteqn{\sum_{k\in\kappa\Zodd} \|u_k\|_{L^q(\Gamma)}^{q'}}\\
&\leq& C_{GN}^{q'} \left(\sum_{k\in\kappa\Zodd} |k|^{-1} \|u_k'\|_{L^2(\Gamma)}^2\right)^\frac{q-2}{4(q-1)} \left(\sum_{k\in\kappa\Zodd} |k|\|u_k\|_{L^2(\Gamma)}^2\right)^\frac{q+2}{4(q-1)} \left(\sum_{k\in\kappa\Zodd} |k|^\frac{-2}{q-2}\right)^\frac{q-2}{2(q-1)} \\ 
&\leq & C_{GN}^{q'} \|u\|_H^{q'} C 
\end{eqnarray*}
where $C=\left(\sum_{k\in\kappa\Zodd} |k|^\frac{-2}{q-2}\right)^\frac{q-2}{2(q-1)}<\infty$ since $2<q<4$. The finishes the proof of the embedding result.
\end{proof}

\begin{proof}[Proof of Lemma~\ref{ConccompLemma}.] It is enough to prove the result in the case $q=2$. Further, it suffices to treat the case where $\tilde q\in (2,3)$. Once this has been accomplished, the remaining case of $\tilde q$ can then be obtained by H\"older-interpolating the $L^{\tilde q}$-norm between two values $q_1\in (2,3)$ and $q_2\in (\tilde q,\infty)$ and using from the previous case that the $L^{q_1}$-norm of the sequence $(u_n)_{n\in\N}$ tends to $0$ while its $L^{q_2}$-norm stays bounded since it is controlled by the $\mathcal{H}$-norm, see Lemma~\ref{specific_values}.

If we let $2<s<q<4$ and set $\theta:= \frac{s-2}{q-2}\cdot\frac{q}{s}$ then H\"older interpolation yields
$$
\|u\|_{L^s(\Gamma_n\times\T_T)} \leq \|u\|_{L^2(\Gamma_n\times\T_T)}^{1-\theta} \|u\|_{L^q(\Gamma_n\times\T_T)}^\theta.
$$
If we make the choice $s=\frac{4(q-1)}{q}$ then $s$ varies in $(2,3)$ if $q$ varies in $(2,4)$ and $s<q$. Moreover, $\theta=\frac{2}{s}$ so that after taking it to the power $s$, the previous inequality becomes
$$
\|u\|_{L^s(\Gamma_n\times\T_T)}^s \leq \|u\|_{L^2(\Gamma_n\times\T_T)}^{(1-\theta)s} \|u\|_{L^q(\Gamma_n\times\T_T)}^2.
$$
Summing over $n$ yields 
\begin{equation} \label{final}
\|u\|_{L^s(\Gamma\times \T_T)}^s \leq \sup_{n\in\Z}\|u\|_{L^2(\Gamma_n\times\T_T)}^{(1-\theta)s} \sum_{n\in\Z} \|u\|_{L^q(\Gamma_n\times\T_T)}^2.
\end{equation}
If we use this inequality for the sequence $(u_n)_{n\in\N}$ form the assumption of the lemma, then the first term of the right-hand side tends to $0$. Thus, we need to find a bound for the second term. Since $q\in (2,4)$ we can use Lemma~\ref{embedding_suboptimal} in the following argument. We take a cut-off function $\phi_n:\Gamma\to [0,1]$, $n\in\Z$ with the properties 
$$
\phi_n(x) = \left\{
\begin{array}{ll}
1 & \mbox{ in } \Gamma_n, \vspace{\jot} \\
0 & \mbox{ in } \Gamma_m \mbox{ with } |m-n|\geq 2, \vspace{\jot}\\
\in [0,1] & \mbox{ in } \Gamma_m \mbox{ with } |m-n|=1.
\end{array}
\right.
$$
We can assume that $\phi$ grows linearly and that $\phi\in H^1_c(\Gamma)_{\text{symm}}$ with $0\leq |\phi'|\leq 1$ on $\Gamma$. Then 
\begin{align*}
\sum_{n\in \Z} \|u\|_{L^q(\Gamma_n\times\T_T)}^2 & \leq \sum_{n\in\Z} \|u\phi_n \|_{L^q(\Gamma\times\T_T)}^2 \\
& \leq C \sum_{n\in \Z} \|u\phi_n\|_H^2 \mbox{ by Lemma~\ref{embedding_suboptimal} } \\
& = C \sum_{n\in \Z} \sum_{k\in \kappa\Zodd} |k|\|u_k\phi_n\|^2_{L^2(\Gamma)} + \frac{1}{|k|} \|(u_k\phi_n)'\|^2_{L^2(\Gamma)} \\
& \leq \sum_{k\in \kappa\Zodd} 3|k|\|u_k\|^2_{L^2(\Gamma)} + \sum_{n\in\Z}\sum_{k\in\kappa\Zodd}\frac{2}{|k|} (\|u_k' \phi_n\|^2_{L^2(\Gamma)} + \|u_k\phi_n'\|^2_{L^2(\Gamma)}) \\
& \leq C\sum_{k\in \kappa\Zodd} \bigl(3|k|+\frac{4}{|k|}\bigr) \|u_k\|_{L^2(\Gamma)}^2 + \frac{6}{|k|}\|u_k'\|_{L^2(\Gamma)}^2 \\
& \leq \tilde C \|u\|_H^2 \leq \bar C \|u\|_{\mathcal H}^2.
\end{align*}
Applying these inequalities to the sequence $(u_n)_{n\in\N}$ from the assumption of the lemma and inserting it into \eqref{final} establishes the claim that $\|u_n\|_{L^s(\Gamma\times\T)}\to 0$ as $n\to\infty$. This finishes the proof of the concentration compactness lemma.
\end{proof}

\section*{Acknowledgment}
	Funded by the Deutsche Forschungsgemeinschaft (DFG, German Research Foundation) – Project-ID 258734477 – SFB 1173. 

\bibliographystyle{plain}

\bibliography{breatherbib}

\begin{thebibliography}{10}

\bibitem{ab_kaup_newell_segur:73}
M.~J. Ablowitz, D.~J. Kaup, A.~C. Newell, and H.~Segur.
\newblock Method for solving the sine-{G}ordon equation.
\newblock {\em Phys. Rev. Lett.}, 30:1262--1264, 1973.

\bibitem{arioli_gazzola:96}
G.~Arioli and F.~Gazzola.
\newblock Periodic motions of an infinite lattice of particles with nearest
  neighbor interaction.
\newblock {\em Nonlinear Anal.}, 26(6):1103--1114, 1996.

\bibitem{Birnir}
Bj\"{o}rn Birnir, Henry~P. McKean, and Alan Weinstein.
\newblock The rigidity of sine-{G}ordon breathers.
\newblock {\em Comm. Pure Appl. Math.}, 47(8):1043--1051, 1994.

\bibitem{BlaSchneiChiril}
Carsten Blank, Martina Chirilus-Bruckner, Vincent Lescarret, and Guido
  Schneider.
\newblock Breather solutions in periodic media.
\newblock {\em Comm. Math. Phys.}, 302(3):815--841, 2011.

\bibitem{brezis_coron}
Ha\"{\i}m Br\'{e}zis and Jean-Michel Coron.
\newblock Periodic solutions of nonlinear wave equations and {H}amiltonian
  systems.
\newblock {\em Amer. J. Math.}, 103(3):559--570, 1981.

\bibitem{brezis_coron_nirenberg}
Ha\"{\i}m Br\'{e}zis, Jean-Michel Coron, and Louis Nirenberg.
\newblock Free vibrations for a nonlinear wave equation and a theorem of {P}.
  {R}abinowitz.
\newblock {\em Comm. Pure Appl. Math.}, 33(5):667--684, 1980.

\bibitem{Denzler}
Jochen Denzler.
\newblock Nonpersistence of breather families for the perturbed sine {G}ordon
  equation.
\newblock {\em Comm. Math. Phys.}, 158(2):397--430, 1993.

\bibitem{gt}
David Gilbarg and Neil~S. Trudinger.
\newblock {\em Elliptic partial differential equations of second order}.
\newblock Classics in Mathematics. Springer-Verlag, Berlin, 2001.
\newblock Reprint of the 1998 edition.

\bibitem{gilgschn}
Steffen Gilg, Dmitry Pelinovsky, and Guido Schneider.
\newblock Validity of the {NLS} approximation for periodic quantum graphs.
\newblock {\em NoDEA Nonlinear Differential Equations Appl.}, 23(6):Art. 63,
  30, 2016.

\bibitem{hirschreichel}
Andreas Hirsch and Wolfgang Reichel.
\newblock Real-valued, time-periodic localized weak solutions for a semilinear
  wave equation with periodic potentials.
\newblock {\em Nonlinearity}, 32(4):1408--1439, 2019.

\bibitem{hofer}
Helmut Hofer.
\newblock On strongly indefinite functionals with applications.
\newblock {\em Trans. Amer. Math. Soc.}, 275(1):185--214, 1983.

\bibitem{james_breathers:09}
Guillaume James, Bernardo S\'{a}nchez-Rey, and Jes\'{u}s Cuevas.
\newblock Breathers in inhomogeneous nonlinear lattices: an analysis via center
  manifold reduction.
\newblock {\em Rev. Math. Phys.}, 21(1):1--59, 2009.

\bibitem{komornik}
Vilmos Komornik.
\newblock Uniformly bounded {R}iesz bases and equiconvergence theorems.
\newblock {\em Bol. Soc. Parana. Mat. (3)}, 25(1-2):139--146, 2007.

\bibitem{kowalczyk_et_al}
Micha\l Kowalczyk, Yvan Martel, and Claudio Mu\~{n}oz.
\newblock Nonexistence of small, odd breathers for a class of nonlinear wave
  equations.
\newblock {\em Lett. Math. Phys.}, 107(5):921--931, 2017.

\bibitem{mackay_aubry}
R.~S. MacKay and S.~Aubry.
\newblock Proof of existence of breathers for time-reversible or {H}amiltonian
  networks of weakly coupled oscillators.
\newblock {\em Nonlinearity}, 7(6):1623--1643, 1994.

\bibitem{Maier}
Daniela Maier.
\newblock Construction of breather solutions for nonlinear {K}lein-{G}ordon
  equations on periodic metric graphs.
\newblock {\em J. Differential Equations}, 268(6):2491--2509, 2020.

\bibitem{mandel_scheider_breather}
Rainer Mandel and Dominic Scheider.
\newblock Variational methods for breather solutions of nonlinear wave
  equations.
\newblock {\em Nonlinearity}, 34(6):3618--3640, 2021.

\bibitem{nehari1}
Zeev Nehari.
\newblock On a class of nonlinear second-order differential equations.
\newblock {\em Trans. Amer. Math. Soc.}, 95:101--123, 1960.

\bibitem{nehari2}
Zeev Nehari.
\newblock Characteristic values associated with a class of non-linear
  second-order differential equations.
\newblock {\em Acta Math.}, 105:141--175, 1961.

\bibitem{pankov}
Alexander Pankov.
\newblock Periodic nonlinear {S}chr\"{o}dinger equation with application to
  photonic crystals.
\newblock {\em Milan J. Math.}, 73:259--287, 2005.

\bibitem{pankov_graph}
Alexander Pankov.
\newblock Nonlinear {S}chr\"{o}dinger equations on periodic metric graphs.
\newblock {\em Discrete Contin. Dyn. Syst.}, 38(2):697--714, 2018.

\bibitem{sch_pel}
Dmitry Pelinovsky and Guido Schneider.
\newblock Bifurcations of standing localized waves on periodic graphs.
\newblock {\em Ann. Henri Poincar\'{e}}, 18(4):1185--1211, 2017.

\bibitem{plum_reichel}
Michael Plum and Wolfgang Reichel.
\newblock A breather construction for a semilinear curl-curl wave equation with
  radially symmetric coefficients.
\newblock {\em Journal of Elliptic and Parabolic Equations}, 2(1):371--387, Apr
  2016.

\bibitem{reed_simon}
Michael Reed and Barry Simon.
\newblock {\em Methods of modern mathematical physics. {I}. {F}unctional
  analysis}.
\newblock Academic Press, New York-London, 1972.

\bibitem{reed_simon_4}
Michael Reed and Barry Simon.
\newblock {\em Methods of modern mathematical physics. {IV}. {A}nalysis of
  operators}.
\newblock Academic Press [Harcourt Brace Jovanovich, Publishers], New
  York-London, 1978.

\bibitem{scheider_breather}
Dominic Scheider.
\newblock Breather solutions of the cubic {K}lein-{G}ordon equation.
\newblock {\em Nonlinearity}, 33(12):7140--7166, 2020.

\bibitem{segur_kruskal}
Harvey Segur and Martin~D. Kruskal.
\newblock Nonexistence of small-amplitude breather solutions in {$\phi^4$}
  theory.
\newblock {\em Phys. Rev. Lett.}, 58(8):747--750, 1987.

\bibitem{SW}
Andrzej Szulkin and Tobias Weth.
\newblock The method of {N}ehari manifold.
\newblock In {\em Handbook of nonconvex analysis and applications}, pages
  597--632. Int. Press, Somerville, MA, 2010.

\bibitem{Willem}
Michel Willem.
\newblock {\em Minimax theorems}, volume~24 of {\em Progress in Nonlinear
  Differential Equations and their Applications}.
\newblock Birkh\"{a}user Boston, Inc., Boston, MA, 1996.

\end{thebibliography}

\end{document}